\documentclass[smallextended,referee,envcountsect,]{svjour3}
\usepackage[caption=false]{subfig}% Support for small, `sub' figures and tables
\usepackage{amsmath,amssymb,amsfonts}
\usepackage{graphicx}
\usepackage{color}
\usepackage{subfig}
\usepackage{float}                                                          
\usepackage{multirow}
\usepackage[ruled,lined,linesnumbered]{algorithm2e}
\usepackage{enumerate}
\usepackage{enumitem}

\def\dist{\textup{dist}}

\newcommand{\norm}[1]{\left\lVert#1\right\rVert}

\newcommand{\Rm}{\mathbb{R}^m}
\newcommand{\Rn}{\mathbb{R}^n}
\newcommand{\Rmn}{\mathbb{R}^{m \times n}}

\newcommand{\R}{\mathbb{R}}

\newcommand{\cN}{\mathcal{N}}
\newcommand{\cR}{\mathcal{R}}

\newtheorem{obs}{Remark}
\newtheorem{defi}{Definition}
\newtheorem{teo}{Theorem}
\newtheorem{lema}{Lemma}

\newtheorem{prop}{Proposition}
\newtheorem{hip}{Assumption}
\usepackage{graphicx}
\usepackage{hyperref}

\begin{document}
	
\title{Convergence analysis of Levenberg-Marquardt method with Singular Scaling for nonzero residual nonlinear least-squares problems\footnote{This work is supported by FAPESC (Fundação de Amparo à Pesquisa e Inovação do Estado de Santa Catarina) [grant number 2023TR000360]}}

%\subtitle{}
\titlerunning{Convergence analysis of LMMSS for nonzero residual NLS problems}

\author{Rafaela Filippozzi \and Everton Boos \and  Douglas S. Gon\c{c}alves \and Ferm\'{i}n S. V. Baz\'{a}n }

\institute{Rafaela Filippozzi \and Everton Boos  \and  Douglas S. Gon\c{c}alves (Corresponding author) \and Ferm\'{i}n S. V. Baz\'{a}n   \at
              Department of Mathematics, Federal University of Santa Catarina, \\
              Florian\'{o}polis, 88040900, SC, Brazil\\
			  rafaela.filippozzi@posgrad.ufsc.br \and  everton.boos@ufsc.br \and  douglas.goncalves@ufsc.br \and  fermin.bazan@ufsc.br  
}

\date{Received: date / Accepted: date}
%The correct dates will be entered by the editor.

\maketitle

	\begin{abstract}
	Recently, a Levenberg-Marquardt method with Singular Scaling matrix, called LMMSS, was proposed and successfully 
	applied in parameter estimation  in heat conduction problems, 
	where the choice of suitable singular scaling matrix resulted in better quality approximate solutions than those of the classical Levenberg-Marquardt. 
%In \cite{ref1}, a local convergence analysis was presented for the zero residual case. 
%	Regularized seminorm Gauss-Newton methods, also called Le\-ven\-berg--Marquardt (LM) methods with singular scaling, have been previously studied in the literature and have proven to be effective in the solution of nonlinear equations with nonisolated zeros. By a suitable choice of the scaling matrix such methods deliver solutions which are physically more sound than the classic LM method. 
	In this paper, we study convergence properties of LMMSS when applied to \emph{nonzero} residual nonlinear least-squares problems. 
	We show that the local convergence of the iterates depends both on the control of the gradient linearization error and on a suitable choice of the regularization parameter. 
	%In specific settings, superlinear convergence can still be achieved but, in general, even linear convergence can be ensured only when a combined measure of nonlinearity and residual size is small enough.
	Incidentally, we show that the rate of convergence  is dictated by
a measure of nonlinearity and residual size, so that if such a measure goes to
zero quickly enough, the convergence can be superlinear, otherwise, in general,
we show that not even linear convergence can be expected if such a measure is
not small enough. 
%	 and  address its convergence analysis to stationary points of the least-squares function. Local convergence is established under completeness and local error bound assumptions. 
%	We show  that the choice of the regularization parameter is dictated by the linearization error of the gradient, which is influenced by the nonlinearity and the size of the residual. In particular, if such an error goes to zero fast enough, we prove that the distance of an iterate to the set of stationary points goes to zero superlinearly. 
Additionally, we propose a globalized version of the method and prove that any limit point of the generated sequence is stationary for the least-squares function. 
%Some examples illustrate how the method behaves in nonlinear problems with nonzero residual.
Some examples are provided to illustrate our theoretical results.
	\end{abstract}
	

	\section{Introduction}

	In this study we investigate convergence properties of the Levenberg--Marquardt method with singular scaling (LMMSS) from \cite{Boos2024} when applied to the nonlinear least-squares (NLS) problem:
%	We consider the following nonlinear least-squares (NLS) problem:
	\begin{equation}\label{prob1}
		\min_{x\in \mathbb{R}^n} \quad \phi(x) := \frac{1}{2}\|F(x)\|^2,
	\end{equation}
	where $F: \mathbb{R}^n \rightarrow \mathbb{R}^m$ is twice continuously differentiable and $\| \cdot \|$ denotes the Euclidean norm. 
	In particular, we  are interested on the \emph{overdetermined} case, where $m \geq n$.
	
	Unlike previous local convergence analyses \cite{Dennis_Schnabel,FanYuan,Gratton2007,Yamashita}, including \cite{Boos2024}, 
	we assume neither zero residual at a solution of \eqref{prob1} nor full rank of the Jacobian at such a point. 	
	%Regularization strategies for nonzero-residual ill-posed problems arising from specific applications, 
	In applied contexts, such as data fitting, parameter estimation, experimental design, and imaging problems  \cite{bellavia2018elliptical,barz2015nonlinear,cornelio2011regularized,deidda2014regularized,henn2003levenberg,landi2017limited}, 
	to name a few, admitting a nonzero residual is essential for achieving meaningful solutions. 
	
%	In this theoretical study, aiming at solving \eqref{prob1}, we shall investigate the following iteration:
The LMMSS iteration is defined as follows \cite{Boos2024}:
	\begin{align}\label{s1}
		(J_k^TJ_k + \lambda_k L^TL) d_k &= -J_k^TF_k\\
		x_{k+1} &= x_k + \alpha_k d_k,  \label{s2}
	\end{align}
	where $F_k := F(x_k) \in \Rm$, $J_k := J(x_k) \in \Rmn$ is the Jacobian of $F$ at $x_k$, $\lambda_k > 0$, $\alpha_k$ is the step size and $L^TL$, referred to as  \emph{scaling matrix},  is allowed to be singular. We refer to the iteration \eqref{s1}--\eqref{s2} as the Levenberg-Marquardt method with Singular Scaling (LMMSS). When $L^T L = I$ we retrieve the \emph{classic} Levenberg-Marquardt method (LMM). 
	
%		##### revisao LM classico, full rank, error bound e por fim residuo nao

In general, problem \eqref{prob1} is a nonconvex optimization problem % for which the global minimum is not known, opposed to the case of zero residual. 
%Thus, 
and we will limit our attention to stationary points of $\phi$. The set of stationary points will be denoted by 
\[
X^* = \{x \in \mathbb{R}^n \mid \nabla \phi(x) = J(x)^TF(x) = 0\},
\]
and assume that $X^* \neq \emptyset$. We are particularly interested in the case of nonisolated stationary points, with a possible change (decrease) in the rank of the Jacobian as the generated sequence of iterates $\{x_k\}$ approaches the set $X^*$.

Given a starting point sufficiently close to $x^* \in X^*$, we are interested in the convergence analysis of the sequence generated by the local LMMSS iteration (with $\alpha_k=1$). Our main contribution consists in establishing local convergence results for LMMSS based on an error bound condition upon $\nabla \phi$ \emph{without} requiring zero residual neither full rank of the Jacobian at stationary points.

	Previous research predominantly focused on the zero residual case  (i.e $\exists\text{ } x \in \Rn$ such that $F(x)=0$) or in cases where the Jacobian of $F$ has full rank  at a solution  \cite{Dennis1977,Dennis_Schnabel,Levenberg1944,Marquardt1963,Yamashita,Boos2024}. 
	In \cite{Dennis_Schnabel,Dennis1977}, local convergence of LMM for the NLS problem was established assuming full rank of the Jacobian at the solution, and the nonzero residual case was handled by imposing a condition on the residual size and nonlinearity. Under such condition, and assuming the sequence of regularization parameters $\lambda_k$ is bounded away from zero, it was proved that the iterates converge linearly to the solution. In the zero residual case, by choosing the regularization parameter proportional to the norm of the gradient, such convergence was proved to be quadratic. 
	
	The seminal work \cite{Yamashita} showed the local convergence of LMM for systems of nonlinear equations under an error bound condition upon the norm of the residual, without relying on the assumption of the full rank of the Jacobian. 
	Assuming $\lambda_k = \| F(x_k) \|^2$, the authors established that the distance of the iterates to the solution set converges quadratically to zero. 
	Later, \cite{FanYuan} improved such result and showed that quadratic convergence is still attained for $\lambda_k = \| F(x_k) \|^{1+r}, r \in [0,1]$. 
	Furthermore, they also showed that the sequence of iterates itself converges quadratically to some point in the solution set. 

	More recently, a few papers have focused on establishing local convergence of LMM for \emph{nonzero} residual NLS problems (i.e $\forall x \in \Rn$,  $F(x) \ne 0$). 
	In \cite{behling2019local}, local convergence was established using an error bound on the gradient of $\phi$ and considering a possible change in the rank of Jacobian around a stationary point. In \cite{Bergou2020}, local convergence of LMM in the nonzero residual case was also established, but using a different error bound condition based on the distance between the residual vector at $x$ and the one at $\bar{x}$, the stationary point closest to $x$.

%	Here we consider NLS problems where the set of zeros of $F$ is empty  (which we call a \emph{nonzero residual} NLS problem) and its Jacobian is allowed to be rank-deficient. 
%
%	In such a case, our interest is to identify the stationary points of $\phi$, i.e., the elements of the set $X^* = \{x \in \mathbb{R}^n \mid \nabla \phi(x) = J(x)^TF(x) = 0\}$, assuming that $X^* \neq \emptyset$. Throughout the text we shall use $x^*$ to denote an arbitrary element of $X^*$. 	

%	Regularization strategies for nonzero-residual ill-posed problems arising from specific applications, such as data fitting, parameter estimation, experimental design, and imaging, have been considered in \cite{bellavia2018elliptical,barz2015nonlinear,cornelio2011regularized,deidda2014regularized,henn2003levenberg,landi2017limited}, where admitting a nonzero residual is essential for achieving meaningful solutions. 

%	Convergence analysis of the classic LMM in the nonzero residual scenario has received recent attention in the literature \cite{behling2019local,Bergou2020}.

%	##### revisao LM com semi-norma L, motivacao, aplicacoes 
	Nevertheless, all the works mentioned above considered the classic LMM where $L^T L = I$. 
	Perhaps \cite{Boos2024} was the first work to consider local convergence analysis of LMMSS (for $L^T L \ne I$) under an error bound condition, but considering the zero residual case. The motivation to use singular matrices in the form $L^TL$ in LMMSS iteration \eqref{s1}--\eqref{s2} comes from the general form of Tikhonov regularization \cite{Tikhonov95} for linear systems $Ax = b$, i.e.,
\begin{equation*}
x_\lambda = (A^TA + \lambda L^TL)^{-1} A^Tb.
\end{equation*}
%In this case, the literature indicates that for discrete ill-posed problems with smooth solution, the choice of $L$ as a discrete version of derivative operators can lead to significant improvements in the generated approximations \cite{Hansen1998,Hansen2010}. 
The use of a singular scaling matrices allows us to use seminorm regularizers which can promote specific properties (inherent of the expected solution) in the approximate solution generated by method. Like the $\ell_1$-norm is used to promote sparsity, if the aimed solution is expected to be smooth, then the use of $L$ as a discrete version of derivative operator allows us to induce such smoothness in the approximate solution. This approach proved to be successful in ill-posed linear inverse problems \cite{Hansen1998,Hansen2010,Buccini2017}. 

Nonlinear inverse problems can also benefit from the more flexible choice of $L$ in LMMSS. 
For instance, in \cite{Boos2024} LMMSS was successfully used to estimate the 2D perfusion coefficient in a bioheat model. 
%and the LMMSS solution compared with the one obtained by the classic LMM. 
Figure~\ref{fig:exemplosingularscaling} shows the surface plots of the 2D perfusion coefficient reconstructions: the picture on the left plots the expected solution, the one on the middle shows the reconstruction obtained by the classic LMM and the one on the right corresponds to the reconstruction obtained by LMMSS with a specific choice of $L=\mathcal{L}_3$, where $\mathcal{L}_3$ is based on a discrete version of the third order derivative operator. 
Both methods stopped after dozens of iterations based on the discrepancy principle as stopping criterion \cite{Morozov1984}. 
For this problem, it is noteworthy that although the temperature reconstruction error was the same for both methods (almost the same value of $\phi$), 
LMMSS provided an approximate solution closer to the expected one.
For more details, see \cite[Section 5.1]{Boos2024}). 

\begin{figure}[h]
\centering
\includegraphics[scale=0.37]{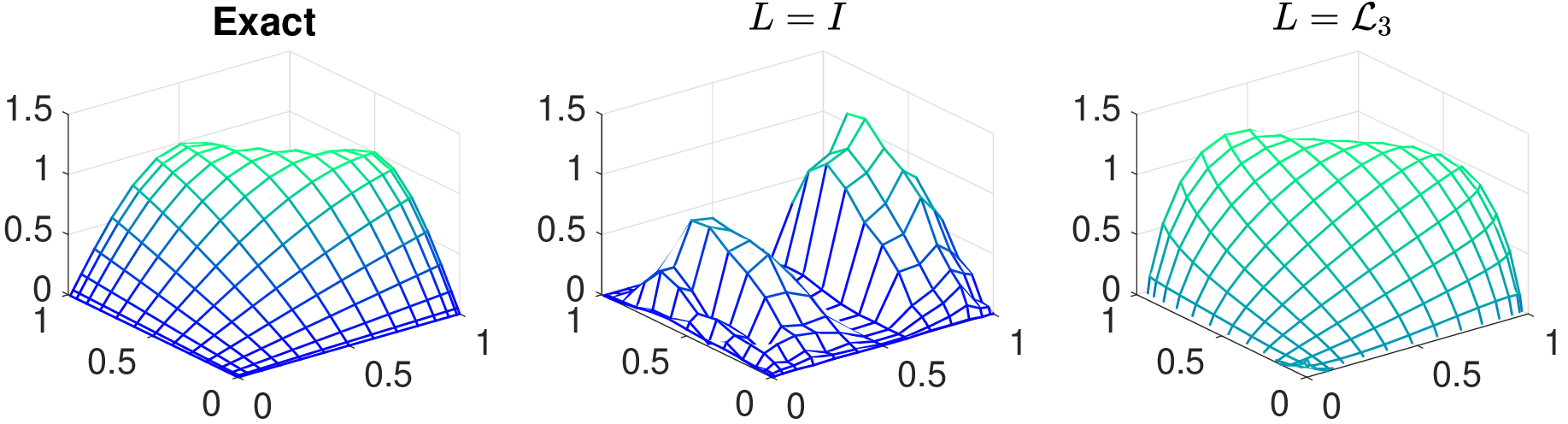}
\label{fig:exemplosingularscaling}
\caption{The picture on the left plots the expected solution, the one on the middle shows the reconstruction obtained by the classic LMM and the one on the right corresponds to the reconstruction obtained by LMMSS with a specific choice of $L$.  Image extracted from \cite[Figure~3]{Boos2024}.}
\end{figure}

The above example illustrates the potential of LMMSS when a suitable $L$ is chosen. 
Observe that  $L$ defines the seminorm regularizer and must be chosen according to the problem at hand, based on known properties of the expected solution. 
As we already mentioned, if the solution (of the original problem, whose discretization gave rise to the nonlinear least-squares problem) is expected to be smooth, then the use of $L$ as a discrete version of derivative operators may be a good choice. But other examples are available in the literature, for instance, see \cite{Buccini2017,pes2020,Pes2022} and references therein.

% #### qual a nossa contribuicao ?????

As far as we know, convergence analysis of LMMSS in the \emph{nonzero residual} scenario has not yet been addressed in the literature. 
This is precisely the purpose of this work: for the LMMSS iteration defined in \cite{Boos2024} we aim at establishing local convergence rates to stationary points of the least-squares function in the case of \emph{nonzero} residual under an error bound condition upon $\nabla \phi$, following the lines of the analysis in \cite{behling2019local}. Of course, by considering a possibly singular $L^T L$ in iteration \eqref{s1}--\eqref{s2} makes the analysis more challenging, for the iteration itself to be well defined depends on a condition on the null spaces of $L$ and the Jacobian of $F$. 

Additionally, we propose a globalization for the LMMSS iteration where the step-size $\alpha_k$ is determined by a line-search satisfying an Armijo condition 
and the search direction is safeguarded by the classic LMM direction. For the resulting algorithm, we prove that any limit point of the generated sequence is stationary.

The paper is organized as follows. 
In Section~\ref{sec:assump}, we gather some necessary assumptions to show that LMMSS iteration is well-defined and for the subsequent analysis. 
In Section~\ref{sec:auxresu}, we present some technical results and discuss the local convergence analysis in Section~\ref{sec:local}. 
As we shall see, the convergence rate is dictated by a measure of nonlinearity and residual size. 
If such a measure goes to zero fast enough, the convergence can be superlinear but, in general, 
we cannot expect even linear convergence if such measure is not small enough. 
%We show that LMMSS iteration is well-defined and provide the local convergence results which are categorized into two cases based on the behavior of the rank of the Jacobian around a stationary point (constant and diminishing rank). 
%Section~\ref{sec:global} is devoted to the global convergence analysis, where we prove that, regardless the starting point, for LMMSS with an Armijo line-search scheme, every limit point of the generated sequence is a stationary point of the sum-of-squares function. 
In Section~\ref{sec:global}, a globalized LMMSS algorithm is proposed and its global convergence established.
To close this theoretical study, some examples are given in Section~\ref{sec:examples} to illustrate the local and global behavior of the algorithm in nonzero residual nonlinear least-squares problems. Final considerations are drawn in Section~\ref{sec:conclusion}. 

\section{Assumptions}\label{sec:assump}

Recall that throughout this manuscript, we assume that $X^* \ne \emptyset$. 
We shall denote a generic element of $X^*$ by $x^*$. 
Given $x$, we denote by $\bar{x}$ an element of $X^*$ such that $ \|x - \bar{x}\| = \text{dist}(x, X^*) $. 
Also, given $x^* \in X^*$, we may use the shorthand notation $F_* = F(x^*)$ and $J_* = J(x^*)$. 
The nullspace of $A$ is denoted by $\cN(A)$ and $\cR(A)$ denotes its range. 

%Assumption~\ref{Hip_nullJLnew} guarantees that $d_k$ solution of \eqref{s1} exists and is unique.

%\begin{hip}\label{Hip_nullJLnew}
%The matrix $L\in \R^{p\times n}$ has $\text{rank}(L) = p \leq n \leq m$ and, given $\Omega \subset \Rn$, there exist $\gamma >0$ such that, for every $x \in \Omega$
%\begin{equation}\label{comp-condition}
%	\norm{J(x) v}^2 + \norm{L v} ^2 \geq \gamma \norm{v}^2, \quad \forall v \in \R^n.
%\end{equation}
%\end{hip}

\begin{defi}\label{def:ucomp}
We say that the \emph{uniform completeness condition} for problem \eqref{prob1} with respect to $L$ holds at $\Omega \subset \Rn$, 
if  $L\in \R^{p\times n}$ has $\text{rank}(L) = p \leq n \leq m$ and there exist a constant $\gamma >0$ such that, for every $x \in \Omega$
\begin{equation}\label{comp-condition}
	\norm{J(x) v}^2 + \norm{L v} ^2 \geq \gamma \norm{v}^2, \quad \forall v \in \R^n.
\end{equation}
\end{defi}

It is worth noting that condition \eqref{comp-condition} is equivalent to 
\begin{equation}\label{kernelJL}
\mathcal{N} (J(x)) \cap \mathcal{N} (L) = \{ {\bf 0} \}.
\end{equation}
Condition~\eqref{kernelJL} is often referred to as a \textit{completeness condition} and comes from the literature of linear inverse problems \cite{Morozov1984,engl}, where $J(x) = A, \forall x \in \Rn$. Observe that it clearly holds when $J(x)$ is of full rank, regardless of the choice of $L$, as it occurs in some applications involving physical problems \cite[Section~5.1]{Boos2024}. 
%In fact, in real-life applications involving ill-posed problems, $L$ is often used to enforce smoothness in the calculated solution, and, even if both $L$ and $J(x)$ are rank deficient, it is very unlikely that vectors in $\mathcal{N} (L)$ belong to $\mathcal{N} (J(x))$~\cite[Chapter~8]{Hansen2010}. 
%That is, in practical applications, the assumption is quite reasonable.

In the context of LMMSS iteration, if the iterates $x_k$ lie in a set $\Omega \subset \Rn$ where the uniform completeness condition holds, 
then $d_k$ solution of \eqref{s1} exists and is unique because \eqref{comp-condition} implies $J_k^T J_k + \lambda_k L^T L$ is positive definite for $\lambda_k > 0$.

Next assumption is standard in the Gauss--Newton and Levenberg--Marquardt literature and asks the Jacobian to be locally Lipschitz. 

\begin{hip}\label{Hip_Jdiffnew}
For some $x^* \in X^*$, there exists a constant $\delta \in ]0,1[$ and $L_0 > 0$ such that 
\begin{equation*}
	\|J(x)-J(y)\| \leq L_0 \|x-y\|,
\end{equation*}
for all $x, y \in B(x^*, \delta)$. 
\end{hip}

Assumption~\ref{Hip_Jdiffnew} implies that:
\begin{equation*}
\|J(y)(x-y) - (F(x)-F(y))\| \leq L_1 \|x-y\|^2,  \quad \forall x, y \in B(x^*, \delta)
\end{equation*}
where $L_1 = L_0/2$, that is, the error in the linear approximation of $F(x)$ around $y$ is $O(\| x - y\|^2)$, for $x$ and $y$ in a neighborhood of $x^*$. 

Due to the compactness of the ball $B(x^*, \delta)$, there exist positive constants $L_2$ and $\beta$ such that $\|J(x)\| \leq L_2$ and $\|F(x)\| \leq \beta$ for all $x \in B(x^*, \delta)$. Therefore, since $\|J(x)\|$ is bounded in $B(x^*, \delta)$, by applying the mean value inequality, we can infer that $\|F(x) - F(y)\| \leq L_2\|x - y\|$, for all $x, y \in B(x^*, \delta)$. Additionally, the gradient $\nabla \phi(x) = J(x)^TF(x)$ is Lipschitz in $B(x^*, \delta)$:
\begin{equation}\label{eqhipL3new}
\norm{J(x)^TF(x) - J(y)^TF(y)} \leq L_3\|x - y\|, \quad \forall x, y \in B(x^*, \delta),
\end{equation}
where $L_3=L_2^2+\beta L_0.$

Moreover, notice that for $z \in X^* \cap B(x^*, r)$ and $x, y \in B(x^*, r)$, we have

\begin{equation}\label{eq:auxhiplips}
\begin{aligned}
	\|(J(x) - J(y))^TF(y)\| &= \|(J(x) - J(z) + J(z) - J(y))^TF(y)\|\\
	&\leq \|(J(x) - J(z))^TF(y)\| + \|(J(z) - J(y))^TF(y)\|\\
	&\leq L_0L_2\|x - z\|\|y - z\| + \|J(x)^TF(z)\| \\&+ L_0L_2\|y - z\|^2 + \|J(y)^TF(z)\|.       
\end{aligned}
\end{equation}

\begin{lema}\cite[Lemma 2.1]{behling2019local}\label{lem:L4}
If Assumption \ref{Hip_Jdiffnew} is satisfied, then there exists $\delta \in ]0,1[$ and a constant $L_4>0$ such that
\begin{equation}\label{eq:L4}
	\norm{ \nabla \phi(y) - \nabla \phi(x) - J(x)^T J(x)(y - x) } \leq L_4 \norm{ x - y }^2 + \norm{ (J(x) - J(y))^T  F(y)},
\end{equation}
for all $x, y \in B(x^*, \delta)$.
\end{lema}

Next, we present a local error bound condition assumed in this work. 

\begin{hip}\label{Hip_error_boundnew}
For some $x^* \in X^*$, $\norm{J(x)^TF(x)}$ provides a local error bound at $x^*$, i.e.,  
there exists $\delta \in ]0,1[$ and $\omega >0 $ such that 
\begin{equation*}%\label{error_boundnew}
	\omega\dist (x, X^*) \leq \norm{J(x)^TF(x)}, \quad \forall x \in B(x^*, \delta),
\end{equation*}
where $\dist(x, X^*)= \inf_{z\in X^*}{\norm{x-z}}$.
\end{hip}
Error bound conditions have been extensively studied and used in convergence analysis of LM methods in the last two decades \cite{Yamashita,fan,FanYuan,Fischer2002,behling2012,behling2013,behling2019local,Fischer2024}. 
For zero residual problems, it is common to use $\| F(x) \|$ as an error bound. 
However, for nonzero residual problems, $F(x) \ne 0, \forall x \in \Rn$, and Assumption~\ref{Hip_error_boundnew} uses $\| \nabla \phi(x) \|$ as an error bound instead. 
Section~\ref{sec:examples} presents some examples where such assumption holds. See also the discussion and examples in \cite{behling2019local}. 
It is well-known that Assumption~\ref{Hip_error_boundnew} is suitable to deal with nonisolated stationary points and it is weaker than the assumption of the Jacobian having full rank at $x^* \in X^*$. By the way, $J(x^*)$ having full rank implies Assumption~\ref{Hip_error_boundnew}.
%\begin{obs}
%Since we deal with the nonzero residual case, as in \cite{behling2019local}, the error bound condition is on $\nabla \phi$ rather than $F$. 
%Although the earliest works on this topic can be traced back to the 1950s, error bound conditions have been frequently used in the literature in the last decades  \cite{Yamashita,behling2019local,Bergou2020,Boos2024} for being weaker than other traditional regularity assumptions, such as requiring full-rank of the Jacobian at the solution.
%\end{obs}

From Assumption~\ref{Hip_error_boundnew} and \eqref{eqhipL3new}, we obtain
\begin{equation}\label{eq:associationerrorwithL3}
\omega\text{dist}(x, X^*) \leq \norm{ J(x)^T F(x) } \leq L_3 \text{dist}(x, X^*).    
\end{equation} 

The remaining assumptions focus, as outlined in \eqref{eq:auxhiplips}, on the terms 
$$ 
\norm{ J(x)^T F(z)} \text{ and } \norm{ J(y)^T F(z) }.
$$ 
These terms play a crucial role in controlling the error, as expressed in \eqref{eq:L4}, of the ``incomplete linearization'' of the gradient: 
notice that $J(x)^T J(x)$ is used in \eqref{eq:L4} instead of the Hessian $\nabla^2 \phi(x) = J(x)^T J(x) + S(x)$, with $S(x) = \sum_{i=1}^m F_i(x) \nabla^2 F_i(x)$, that would appear in a first order Taylor approximation of $\nabla \phi$.
%Next assumptions are about the behavior of such errors in a neighborhood of $x^* \in X^*$.

\begin{hip}\label{a3}
For some $x^* \in X^*$, for all $x \in B(x^*, \delta)$ and all $z \in X^* \cap B(x^*, \delta)$, the following inequality holds:
\begin{equation}\label{eq:a3}
	\|J(x)^TF(z)\| \leq \sigma \|x - z\|, 
\end{equation}
with $0 \leq \sigma < \bar\sigma$, where $\bar\sigma$ is a positive constant depending on $\omega$, $L_3$ and the smallest positive eigenvalue of $J(x^*)^T J(x^*)$.
\end{hip}

Explicit expressions for $\bar\sigma$ will be given ahead. 
Assumption~\ref{a3} dates back, at least, to the work of Dennis \cite{Dennis1977}, 
and it was also considered in  \cite{behling2019local} where the authors analyze the local convergence of the standard LM (where $L^T L=I$) in the case of nonzero residual. 
It is a condition to control the linearization error of the gradient. 
In \cite{Dennis1977}, convergence analysis of the classic LMM is given for the nonzero residual case, under the assumption that $J(x^*)^TJ(x^*)$ is nonsingular and $\|J(x)^TF(x^*)\| \leq \sigma \|x - x^*\|$, in a neighborhood of $x^* \in X^*$, for a sufficiently small $\sigma>0$, namely $\sigma < \lambda^* := \lambda_{\min}(J(x^*)^T J(x^*))$. 
Here, because we are not requiring $J(x^*)$ to have full column rank, the set $X^* \cap B(x^*, \delta)$ may be a nonisolated set of stationary points and for this reason Assumption~\ref{a3} is slightly different from that of \cite{Dennis1977}. 

\begin{obs}\label{rem:a3smallres}
Observe that 
\[
\|(J(x)-J(z))^TF(z)\| \leq \| J(x) - J(z) \| \|F(z) \| \leq L_0 \| F(z) \| \| x -z \|.
\]
Thus,  given a bound $\|F(z)\| \leq \beta$ for every $z \in X^* \cap B(x^*, \delta)$, if $L_0 \beta < \bar\sigma$, then \eqref{eq:a3} holds with $\sigma = L_0 \beta$. 
This might occur when the residual is small enough in $X^* \cap B(x^*, \delta)$. 
\end{obs}

\begin{hip}\label{hipotese5paraconvl}
For some $x^* \in X^*$, for all $x \in B(x^*, \delta)$ and all $z \in X^* \cap B(x^*, \delta)$, the following inequality holds:
\begin{equation*}%\label{eq:hip5}
	\|J(x)^TF(z)\| \leq C\|x - z\|^{1+r}, 
\end{equation*}
with $r \in ]0,1]$ and $C\geq 0$.    
\end{hip}

Assumption~\ref{hipotese5paraconvl} is admittedly stronger than Assumption~\ref{a3}, but it might hold in some specific problems. 
%Assumption~\ref{hipotese5paraconvl} was also employed in \cite{behling2019local} for analyzing the local convergence of the classic LMM in the nonzero residual case. 
For example, it is clear that such assumption holds when $F$ is linear or $F(z) = 0$. 
In \cite{behling2019local} the authors showed that when Assumption~\ref{hipotese5paraconvl} holds, LMM local convergence rate can be superlinear under a suitable choice of the LM parameter.    
%Other examples will be given in Section~\ref{sec:examples}. 
In \cite{Dennis_Schnabel}, it is mentioned that $J(x)^T F(z) \approx S(z)(x-z)$ and thus $\sigma$ could be interpreted as a combined measure of nonlinearity and residual size. Observe that $S(z)=0$ is also a sufficient (but not necessary) condition to ensure Assumption~\ref{hipotese5paraconvl}. 
See the examples in Section~\ref{sec:examples} and in the reference \cite{behling2019local}. 

\section{Auxiliary Results}\label{sec:auxresu}

In this section, we organize other fundamental results for the upcoming analysis. 

%%We start by mentioning that \eqref{s1} always has a solution. In fact, \eqref{s1} corresponds to the normal equation $B^T B d = B^T c$ of the augmented system
%%\[
%%Bd := \left[ \begin{array}{c}
%%J_k \\
%%\sqrt{\lambda_k} L
%%\end{array} \right] d = - \left[ \begin{array}{c}
%%F_k \\
%%0
%%\end{array} \right] =: c,
%%\]
%%and the normal equation always has a solution because $B^T c$ lies in the range of $B^T B$. 
%%But it is condition \eqref{comp-condition} that ensures the solution of \eqref{s1} is unique. 
%%Indeed, it is easy to see that condition \eqref{comp-condition} at $x_k$ implies that $J_k^T J_k + \lambda_k L^T L$ is positive definite.

We start with a perturbation lemma that will be important when the rank of $J(x)^TJ(x)$ is constant in a neighborhood of $x^*$.

\begin{lema}\label{lem:banach} \cite[Corollary 3.1]{behling2019local}
Given $\kappa > 1$, if $\text{rank}(J(x)^T J(x)) = \text{rank}(J_*^T J_*) = q \geq 1$, and 
\begin{equation}\label{eq:kappa}
	\norm{J(x)^TJ(x)-J_*^T J_*} \leq \Big( 1 - \frac{1}{\kappa} \Big) \frac{1}{\norm{(J_*^TJ_*)^+}}, 
\end{equation}
where $(J_*^TJ_*)^+$ denote the Moore-Penrose pseudo-inverse, then $$\norm{(J(x)^TJ(x))^+} \leq \kappa \norm{(J_*^TJ_*)^+}.$$
\end{lema}

Now, we will recall some results on the Generalized Singular Value Decomposition (GSVD), 
an important tool in theoretical analysis of the LMMSS direction.

\begin{teo}(GSVD) \cite[p. 22]{Hansen1998}\label{teoGSVD}
Consider the pair $(A, L)$, where $A \in \mathbb{R}^{m \times n}$, $L \in \mathbb{R}^{p \times n}$, $m \geq n \geq p$, $rank(L) = p$, and $\mathcal{N} (A) \cap \mathcal{N}(L) = \{ 0 \}$. 
Then, there exist matrices $U \in \mathbb{R}^{m \times n}$ and $V \in \mathbb{R}^{p \times p}$ with orthonormal columns 
and a nonsingular matrix $X \in \mathbb{R}^{n \times n}$ such that 
\begin{equation*}%\label{GSVD}
	A = U \left[ \begin{array}{cc}
		\Sigma & 0 \\ 
		0 & I_{n-p}
	\end{array}  \right] X^{-1}  \quad \text{and} \quad L = V \left[ \begin{array}{cc}
		M & 0
	\end{array}  \right] X^{-1},
\end{equation*}
with $\Sigma$ and $M$ being the following diagonal matrices:
\begin{equation*}
	\Sigma = \text{diag}(\sigma_{1}, \dots, \sigma_{p}) \in \mathbb{R}^{p \times p} \quad \text{and} \quad M =  \text{diag}(\mu_{1}, \dots, \mu_{p}) \in \mathbb{R}^{p \times p}.
\end{equation*}
Moreover, the elements of $\Sigma$ and $M$ are nonnegative, ordered as follows:
\begin{equation*}
	0 \leq \sigma_{1} \leq \dots \leq \sigma_{p} \leq 1 \quad \text{and} \quad 1 \geq \mu_{1} \geq \dots \geq \mu_{p} >0,
\end{equation*}
and normalized by the relation $\sigma_{i}^{ 2} + \mu_{i}^{ 2} = 1$, for $i = 1, \dots, p$. 
We call the generalized singular value of the pair $(A, L)$ the ratio
\begin{equation*}
	\gamma_{i} = \frac{\sigma_{i}}{\mu_{i}}, \quad i = 1, \dots, p.
\end{equation*}
\end{teo}

\begin{obs}\label{obs:ranks} 
Since $X^{-1}$ is nonsingular and $U$ has orthonormal columns we have 
\[
\text{rank}(A) = \text{rank}(U \left[ \begin{array}{cc}
	\Sigma & 0 \\ 
	0 & I_{n-p}
\end{array}  \right] X^{-1}) = \text{rank}(\left[ \begin{array}{cc}
	\Sigma & 0 \\ 
	0 & I_{n-p}
\end{array}  \right] ) \geq n - p.
\]
\end{obs}

%Such remark is important in situations where the rank of $J(x)$ remains constant, as well as instances where it gradually decreases after a certain number of iterations. In any case, however, the bound on the rank of $J(x)$ remains valid. 

By considering the GSVD of the pair $(J_k, L)$ we can provide a useful characterization of the direction $d_k$ in LMMSS.
In fact, given the GSVD 

\begin{equation*}
J_k = U_k \left[ \begin{array}{cc}
	\Sigma_k & 0 \\
	0 &I_{n-p}
\end{array}  \right] X_k^{-1}  \quad \text{and} \quad L = V_k \left[ \begin{array}{cc}
	M_k & 0
\end{array}  \right] X_k^{-1},
\end{equation*}
where   
\begin{equation}\label{eq:sigmamuigsvd}
(\Sigma_k)_{ii} := \sigma_{i,k} \quad \text{and} \quad (M_k)_{ii} := \mu_{i,k}, \quad i = 1, \dots, p;
\end{equation}
it follows that   
\begin{equation}\label{JmuL}
J_k^TJ_k + \lambda_k L^TL = X_k^{-T} \left[ \begin{array}{cc}
	\Sigma_k^2 + \lambda_k M_k^2 & 0 \\ 
	0 & I_{n-p}
\end{array}  \right] X_k^{-1}.  %\quad \forall k \in \mathbb{N},
\end{equation}

Then, $d_k$ from \eqref{s1} can be expressed as: 
\begin{equation}\label{eq:dkGSVD}
d_k = -X_k \left[ \begin{array}{cc}
	\Gamma_k & 0 \\ 
	0 & I_{n-p}
\end{array}  \right] X_k^T {J_k}^T F_k, %\quad \forall k \in \mathbb{N},
\end{equation}
with $\Gamma_k:=(\Sigma_k^2 + \lambda_k M_k^2)^{-1}.$

The following result provides bounds for the matrix $X_k$.

\begin{lema}\cite[Lemma 2.2]{Boos2024}\label{lem:normXk}
Consider $\{x_k\}$ the sequence generated by LMMSS along with the GSVD of the pair $(J_k,L)$. 
If $\{x_k\} \subset \Omega$ where the uniform completeness condition holds, then 
for each $k \in \mathbb{N}$, we have
\begin{equation} \label{cotaXk}
	\|X_k\| \le \dfrac{1}{\sqrt{\gamma}}.
\end{equation}
\end{lema}

\begin{obs}
By the GSVD for the pair $(J_k,L)$, it follows that 
$$
J_k^TJ_k+ L^TL= X_k^{-T}X_k^{-1}.
$$ 
Therefore, 
\begin{equation}\label{eq:limnormXk-1}
	\norm{X_k^{-1}}^2\leq  \norm{J_k^TJ_k}+ \norm{L}^2 \leq L_2^2+ \norm{L}^2.
\end{equation}
\end{obs}

Now, let us analyze an upper bound for $\norm{\Gamma_k}$. 
From the GSVD for the pair $(J_k, L)$, using the generalized singular values $\gamma_{ i,k} = \sigma_{ i,k}/\mu_{ i,k}$, 
we have:
\begin{equation*}
\sigma_{ i,k}^2 = \frac{\gamma_{ i,k}^2 }{\gamma_{ i,k}^2 + 1} \quad \text{and} \quad \mu_{ i,k}^2 = \frac{1}{ \gamma_{ i,k}^2 + 1}.
\end{equation*}
Thus, we can express $\Gamma_k$ from \eqref{eq:dkGSVD} as:
\begin{equation*}
(\Gamma_k)_{ii} = \frac{1}{ \sigma_{ i,k}^2 + \lambda_k \mu_{ i,k}^2} = \frac{ \gamma_{ i,k}^2+1}{\gamma_{ i,k}^2 + \lambda_k}, \quad i = 1, \dots, p.
\end{equation*}

Therefore, we can bound $\| \Gamma_k \|$ in terms of the generalized singular values $\gamma_{i,k}$ and the parameter $\lambda_k$. To this end, we can use the following lemma.

\begin{lema}\label{lemma_psi}
For the function 
\begin{equation*}
	\psi(\gamma,\lambda)=\dfrac{\gamma^2+1}{\gamma^2+\lambda}, \quad
	\gamma \geq 0, \ \lambda > 0,
\end{equation*}
the following properties hold
\begin{itemize}
	\item[(a)] For $\lambda \in ]0,1[$, the function $\psi (\gamma, \lambda)$ has a unique maximum value attained at 
	\begin{equation*}
		\gamma_{\max} = 0 \quad \text{ where } \quad \psi(\gamma_{\max},\lambda) = \max_{\gamma \geq 0} \psi(\gamma,\lambda)= \frac{1}{\lambda}.
	\end{equation*}
	\item[(b)] For a fixed $\lambda \in [1,+\infty)$, the function $\psi (\gamma, \lambda)$ is nondecreasing and upper-bounded. More precisely,
	\begin{equation*}
		\psi (\gamma, \lambda) \leq 1, \quad \forall \gamma \geq 0.
	\end{equation*}
\end{itemize}
\end{lema}

\begin{proof}
First, note that for any fixed $\lambda$, we have 
\begin{equation}\label{psi_lim_dir}
	\lim_{ \gamma \rightarrow +\infty} \psi (\gamma, \lambda) = 1 = \lim_{ \gamma \rightarrow -\infty} \psi (\gamma, \lambda).
\end{equation}
Now, given $\lambda > 0$, observe that 
\begin{equation}\label{deriv_psi}
	\frac{ \partial \psi}{ \partial \gamma} (\gamma, \lambda) = \frac{ 2(\lambda - 1) \gamma}{(\gamma^2+ \lambda)^2 }.
\end{equation}
\begin{itemize}
	\item[(a)] Consider $\lambda \in ]0,1[$ fixed. As the maximizers of $\psi(\gamma,\lambda)$, for $\gamma \in \R$, must satisfy $\frac{ \partial \psi}{ \partial \gamma} (\gamma, \lambda) = 0,$ from \eqref{deriv_psi}, we have
	\begin{equation*} 
		2(\lambda - 1) \gamma = 0 \quad \Rightarrow \quad \gamma_{\max} = 0,
	\end{equation*}
	Substituting $\gamma_{\max}$ into $\psi(\gamma,\lambda)$, we obtain
	\begin{equation*}
		\psi(\gamma_{\max},\lambda) = \frac{1}{\lambda}.
	\end{equation*}

	\item[(b)] Now, for fixed $\lambda \in [1,+\infty)$, from \eqref{deriv_psi}, we have $\frac{ \partial \psi}{ \partial \gamma} (\gamma, \lambda) > 0$ for $\gamma > 0$. Therefore, $\psi$ is increasing for $\gamma \in [0, +\infty)$, with $\psi(0,\lambda)=1/\lambda \leq 1$. Then, from \eqref{psi_lim_dir}, we conclude that $\psi (\gamma, \lambda) \leq 1$, for all $\gamma \geq 0$. \qed 
\end{itemize}
\end{proof}

We are now ready to prove two key lemmas for the local convergence analysis. 
They show that the norm of LMMSS direction $d_k$ is $O(\dist(x_k,X^*))$, considering two distinct cases. 
Lemma~\ref{lemac1postoconst} addresses the scenario where the rank of the Jacobian around a stationary-point is constant, 
while Lemma~\ref{lemac1postodiminuindo} deals with the case when the rank decreases\footnote{Observe that the rank function is lower semicontinuous thus, when a sequence $x_k$ in a neighborhood of $x^* \in X^*$ approaches $x^*$, the rank of the Jacobian $J(x_k)$ either remains constant or decreases, these are the only two possibilities.}.

\begin{lema}\label{lemac1postoconst}
Suppose that Assumption~\ref{Hip_Jdiffnew} is valid in $B(x^*, \delta)$, 
that the uniform completeness condition holds at $B(x^*, \delta)\setminus X^*$, 
and 
$$
rank(J(x)^TJ(x))=rank(J(x^*)^TJ(x^*))=q\geq 1
$$ 
for all $x\in B(x^*, \delta)$.
If $x_k \in B(x^*, \delta)$, $\lambda_k>0$ and $J_k^T F_k \ne 0$, then there exists $c_1>0$ such that
\begin{equation*}%\label{eq:dk_dist}
	\|d_k\| \leq c_1 \dist (x_k, X^*).
\end{equation*}
\end{lema}

\begin{proof}
Recall that $J_k^T F_k \in \mathcal{R}(J_k^T)= \mathcal{N}(J_k)^\perp  = \mathcal{N}(J_k^T J_k)^{\perp}$, and
\begin{equation}\label{eq:direction}
	(J_k^TJ_k + \lambda_k L^TL) d_k = -J_k^TF_k.     
\end{equation}
First, notice that $d_k\notin \mathcal{N}(J_k)$. In fact, from \eqref{eq:direction}, we have
$$ d_k^T   (J_k^TJ_k + \lambda_k L^TL) d_k = -(J_k d_k)^TF_k,$$
which implies that $\norm{J_k d_k}^2+ \lambda_k \norm{Ld_k}^2 = -(J_k d_k)^TF_k$, 
thus,  if we assume that $d_k \in \mathcal{N}(J_k)$, 
we conclude that $\lambda_k \norm{Ld_k}^2=0$. 
But $\norm{Ld_k}>0$, given that, by condition~\ref{comp-condition}, $\cN(J_k) \cap \cN(L) = 0$, and also $\lambda_k \neq 0$, 
and we have a contradiction. 

We can write $d_k= d_N+ d_R,$  with $ d_N \in \mathcal{N}(J_k)$ and $d_R \in \mathcal{N}(J_k)^\perp =\mathcal{R}(J_k^T)$. 
Since  $d_k\notin \mathcal{N}(J_k)$, we have $d_R \neq 0$.

Thus, from \eqref{eq:direction}
\begin{equation}\label{eq:dkdecomp}
	\begin{aligned}
		-J_k^TF_k &=(J_k^TJ_k + \lambda_k L^TL) d_k\\ 
		&= (J_k^TJ_k + \lambda_k L^TL) (d_N+d_R)\\
		&= J_k^TJ_kd_N + J_k^TJ_kd_R + \lambda_k L^TLd_N+ \lambda_k L^TLd_R \\
		&= J_k^TJ_kd_R + \lambda_k L^TLd_N+ \lambda_k L^TLd_R. 
	\end{aligned}
\end{equation}

When we multiply \eqref{eq:dkdecomp} by $d_N^T$, we get
$$  (J_kd_N)^TJ_kd_R + \lambda_k \norm{Ld_N}^2+ \lambda_k d_N^TL^TLd_R =-(J_kd_N)^TF_k, $$
and as $d_N \in \cN(J_k)$ and $\lambda_k > 0$, we have:
\begin{equation}\label{eq:multdn1}
	\norm{Ld_N}^2 = -(Ld_N)^TLd_R.
\end{equation}

From \eqref{eq:multdn1}, we also conclude that $\norm{Ld_N}^2 \leq \norm{Ld_N}\norm{Ld_R}$, 
and as $d_N \notin {\cal N}(L)$ we obtain
\begin{equation}\label{eq:multdn2}
	\norm{Ld_N}\leq \norm{Ld_R}.
\end{equation}

On the other hand, when we multiply \eqref{eq:dkdecomp} by $d_R^T$, we get
\begin{equation}\label{eq:multdr1}
	\norm{J_kd_R}^2 + \lambda_k (Ld_R)^TLd_N+ \lambda_k \norm{Ld_R}^2 =-(J_kd_R)^TF_k.   
\end{equation}

Using \eqref{eq:multdn1} in \eqref{eq:multdr1}, we are left with
$$ \norm{J_kd_R}^2 + \lambda_k \norm{Ld_R}^2 =-(J_kd_R)^TF_k + \lambda_k \norm{Ld_N}^2.$$

From this and \eqref{eq:multdn2}, we have
\begin{equation}\label{eq:multdr2}
	\norm{J_kd_R}^2 \leq -(J_kd_R)^TF_k.
\end{equation}

Let $J_k= U\Sigma V^T$ be the SVD for $J_k$. As $d_R \in \cN(J_k)^\perp$, $d_R$ can be written as a linear combination of the columns of $V$ corresponding to the nonzero singular values of $J_k$. Therefore,
$$\norm{J_kd_R} = \norm{U\Sigma V^Td_R}\geq s_{q,k}\norm{V^Td_R}= s_{q,k}\norm{d_R}, $$
where $s_{q,k}$ is the smallest positive singular value of $J_k$. 
By using $s_{q,k}\norm{d_R}\leq \norm{J_kd_R} $ in \eqref{eq:multdr2}:
$$s_{q,k}^2\norm{d_R}^2 \leq -(J_kd_R)^TF_k \leq \norm{J_k^TF_k}\norm{d_R},$$
which implies
\begin{equation}\label{eq:limitantedr}
	\norm{d_R} \leq \frac{1}{s_{q,k}^2}\norm{J_k^TF_k}.
\end{equation}

Now let us show that $d_R\notin \cN(L)$. Suppose $d_R\in \cN(L)$, from \eqref{eq:multdn1}, we get $\norm{Ld_N}^2= -d_N^TL^TLd_R =0,$ implying that $d_N\in \cN(L)$ and $d_N\in \cN(J_k)$, which is a contradiction with condition \eqref{kernelJL}. 

By condition \eqref{comp-condition}, there exists $\gamma > 0$ such that 
$$ \gamma\norm{d_N}^2 \leq \norm{J_kd_N}^2 + \norm{Ld_N}^2,$$
which implies 
\begin{equation}\label{eq:complem_DN}
	\norm{d_N}^2 \leq \frac{1}{\gamma}\norm{Ld_N}^2.
\end{equation}

Therefore, from \eqref{eq:multdn2}, \eqref{eq:limitantedr}, \eqref{eq:complem_DN}, and Assumption~\ref{Hip_Jdiffnew},
\begin{equation*}
	\begin{aligned}
		\norm{d_k}^2 &= \norm{d_N+d_R}^2 = \norm{d_N}^2+ 2d_N^Td_R+ \norm{d_R}^2 = \norm{d_N}^2+ \norm{d_R}^2\\
		&\leq  \frac{1}{\gamma}\norm{Ld_N}^2 + \norm{d_R}^2 \leq \frac{1}{\gamma}\norm{Ld_R}^2 + \norm{d_R}^2\\
		&\leq \Big(\frac{\norm{L}^2}{\gamma} +1\Big) \frac{1}{s_{q,k}^4}\norm{J_k^TF_k}^2\\
		&= \Big(\frac{\norm{L}^2}{\gamma} +1\Big) \frac{1}{s_{q,k}^4}\norm{J_k^TF_k - J_*^TF_*}^2 \\
		&\leq \Big(\frac{\norm{L}^2}{\gamma} +1\Big) \frac{1}{s_{q,k}^4}L_3^2 \dist^2(x_k, X^*).
	\end{aligned}
\end{equation*}

Since $J(x_k)^TJ(x_k)$  is continuous, it is clear that for a sufficiently small $\delta>0$, 
condition \eqref{eq:kappa} is satisfied for all $x_k\in B(x^*,\delta)$ and applying Lemma~\ref{lem:banach} we have
$$
\norm{d_k} \leq \Big(\sqrt{\frac{\norm{L}^2}{\gamma}+1}\Big) \frac{\kappa}{\lambda_q^*}L_3 \dist(x_k, X^*),
$$
where we recall that $\lambda_q^*$ is the smallest positive eigenvalue of $J_*^TJ_*$. 
Then define $c_1 = \Big(\sqrt{\frac{\norm{L}^2}{\gamma}+1}\Big) \frac{\kappa}{\lambda_q^*}L_3$ and we have the result.  \qed
\end{proof}

\begin{lema}\label{lemac1postodiminuindo}
Suppose that Assumptions~\ref{Hip_Jdiffnew} and \ref{Hip_error_boundnew} are valid in $B(x^*,\delta)$, for some $\delta \in ]0,1[$ 
and that the uniform completeness condition holds at $B(x^*,\delta) \setminus X^*$. 
If  $rank(J_k)= \ell \geq rank (J_*)= q \geq 1$, and \\
\ \\
(a) $x_k \in B(x^*, \delta) \setminus X^*$, Assumption~\ref{hipotese5paraconvl} is satisfied with $r \in ]0,1]$, $\lambda_k=\norm{J_k^TF_k}^r$,  or,  \\
\ \\
(b) $x_k \in B(x^*, \bar{\delta}) \setminus X^*$, Assumption~\ref{a3} is satisfied with $\sigma < (\sigma_{\min}^*)^2$, 
where $\sigma_{\min}^*$ is the smallest positive singular value of $J(x^*)$ and 
\begin{equation}\label{eq:lambda3}
L_4 \| x_k  - \bar{x}_k \| + \sigma \leq \lambda_k \leq \theta \left(  L_4 \| x_k  - \bar{x}_k \|  + \sigma\right), 
\end{equation}
where $\theta > 1$ is such that $(\sigma_{\min}^*)^2 > \theta \sigma$ and $\bar{\delta} = \min \left\{ \delta, \frac{(\sigma_{\min}^*)^2 - \theta \sigma}{\theta L_4} \right\}$, \\
\ \\
then there exists $c_1>0$ such that $\norm{d_k} \leq c_1 \dist(x_k, X^*)$.
\end{lema}

\begin{proof}
Consider the GSVD for the pair $(J_k,L)$. In this case where $rank(J_k)= \ell \geq q = rank (J_*) \geq 1$ and $rank(L)= p$ with $n-q \leq p$ and $n-\ell\leq p$ (see Remark~\ref{obs:ranks}), we observe that: 
\begin{itemize}
	\item For $1\leq i< n-\ell$, we have $\sigma_{i,k}$=0 and $\mu_{i,k}=1.$ This occurs because of the fact that $rank(J_k)=\ell$ and the relationship between the singular values of the GSVD (see Theorem~\ref{teoGSVD}). Thus, in this case, $(\Gamma_k)_{ii} = \frac{1}{\lambda_k}.$ 
	\item For $n-\ell \leq i <  n-q$, if $\sigma_{i,k} \rightarrow 0$, then $\mu_{i,k} \rightarrow 1$, and we have $(\Gamma_k)_{ii} \rightarrow \frac{1}{\lambda_k}$.
	\item For $n-q \leq  i\leq p$, then $(\Gamma_k)_{ii}= \frac{1}{\sigma_{i,k}^2+\lambda_k \mu_{i,k}^2} < \frac{1}{\sigma_{i,k}^2}\leq \frac{\kappa}{(\sigma^{*}_{\text{min}})^2 },$
	where $\sigma^{*}_{\text{min}}$ is the smallest positive singular value of $J(x^*)$ and $\kappa > 1$ (cf. Lemma~\ref{lem:banach}). 
\end{itemize}

Note that, for $x_k \in B(x^*,\delta)$ and due to the continuity of $\mu_{i,k}$, 
in the case where $n-\ell \leq i <n-q$ we can assume that $\mu_{i,k}> \frac{1}{\sqrt{\kappa}}$ 
and thus, for $\delta$ sufficiently small, we have $(\Gamma_k)_{ii} \leq \frac{\kappa}{\lambda_k}$. 

(a) If $\lambda_k = \| J_k^T F_k \|$, for $\delta$ small enough, $\lambda_k  < \min \{ (\sigma_{\min}^*)^2/\kappa, 1 \}$, and we obtain
\begin{equation}\label{eq:normamatrizdiag}
	\norm{\left[ \begin{array}{cc}
			\Gamma_k & 0 \\ 
			0 & I_{n-p}
		\end{array}  \right]} \leq \max{\Big[\frac{1}{\lambda_k}, \frac{\kappa}{\lambda_k},\frac{\kappa}{(\sigma^{*}_{\text{min}})^2}, 1 \Big]} \leq \frac{\kappa}{\lambda_k}.
\end{equation}

From \eqref{eq:dkGSVD} and \eqref{eq:normamatrizdiag}, we get

\begin{equation}\label{eq:dkGSVDnorm}
	\begin{aligned}
		\norm{d_k} &= \norm{X_k \left[ \begin{array}{cc}
				\Gamma_k & 0 \\ 
				0 & I_{n-p}
			\end{array}  \right] X_k^T {J_k}^T F_k} \\ 
		&= \left\|X_k \left[ \begin{array}{cc}
			\Gamma_k & 0 \\ 
			0 & I_{n-p}
		\end{array}  \right] X_k^T \Big({J_k}^T F_k  - J(\bar{x}_k)^TF(\bar{x}_k)-\right.\\
		&\qquad \left.-J_k^TJ_k(x_k-\bar{x}_k)+J_k^TJ_k(x_k-\bar{x}_k)\Big)\right\| \\
		&\leq   \norm{X_k}^2 \frac{\kappa}{\lambda_k}\norm{{J_k}^T F_k  - J(\bar{x}_k)^TF(\bar{x}_k)-J_k^TJ_k(x_k-\bar{x}_k)}\\&+ \norm{X_k \left[ \begin{array}{cc}
				\Gamma_k & 0 \\ 
				0 & I_{n-p}
			\end{array}  \right] X_k^T J_k^TJ_k(x_k-\bar{x}_k)}.
	\end{aligned}
\end{equation}

Observe from GSVD that 
\begin{equation*}
	\begin{aligned}
		X_k J_k^TJ_k &= X_k X_k^{-T} \left[ \begin{array}{cc}
			\Sigma_k & 0 \\
			0 &I_{n-p}
		\end{array}  \right] U_k^T U_k \left[ \begin{array}{cc}
			\Sigma_k & 0 \\
			0 &I_{n-p}
		\end{array}  \right] X_k^{-1} = \left[ \begin{array}{cc}
			\Sigma_k^2 & 0 \\
			0 &I_{n-p}
		\end{array}  \right]X_k^{-1},
	\end{aligned}
\end{equation*} 
then, as $\Gamma_k:=(\Sigma_k^2 + \lambda_k M_k^2)^{-1}$, 

\begin{equation}\label{eq:dkJtJ}
	\begin{aligned}
		\norm{X_k \left[ \begin{array}{cc}
				\Gamma_k & 0 \\ 
				0 & I_{n-p}
			\end{array}  \right] X_k^T J_k^TJ_k(x_k-\bar{x}_k)} &\leq \norm{X_k} \zeta \norm{X_k^{-1}}\norm{x_k-\bar{x}_k} \\
		&\leq \norm{X_k}\norm{X_k^{-1}}\norm{x_k-\bar{x}_k},
	\end{aligned}
\end{equation}
where this last inequality comes from the fact that 
$$ \norm{\left[ \begin{array}{cc}
		(\Sigma_k^2 + \lambda_k M_k^2)^{-1}\Sigma_k^2 & 0 \\ 
		0 & I_{n-p}
	\end{array}  \right]} \leq  \max\Big\{\max_{1\leq i\leq p}{\Big\{\frac{\sigma_{i,k}^2}{\sigma_{i,k}^2+\lambda_k\mu_{i,k}^2}}\Big\},1\Big\} =: \zeta \leq 1$$ with $\sigma_{i,k}$ and $\mu_{i,k}$ as in \eqref{eq:sigmamuigsvd}.

Therefore, using \eqref{eq:dkGSVD}, \eqref{eq:dkJtJ}, \eqref{eq:L4}, \eqref{eq:limnormXk-1} and %the fact that $\norm{X_k}^2 \leq \frac{1}{\gamma}$ (see 
Lemma~\ref{lem:normXk} in \eqref{eq:dkGSVDnorm}, we have
\begin{equation}\label{eq:dkGSVDnorm2}
	\begin{aligned}
		\norm{d_k} 
		&\leq   \norm{X_k}^2 \frac{\kappa}{\lambda_k}\norm{{J_k}^T F_k  - J(\bar{x}_k)^TF(\bar{x}_k)-J_k^TJ_k(x_k-\bar{x}_k)}\\&+ \norm{X_k}\norm{X_k^{-1}}\norm{x_k-\bar{x}_k} \\
		&\leq \frac{\kappa}{\gamma\lambda_k}L_4\norm{x_k-\bar{x}_k}^2 +\frac{\kappa}{\gamma\lambda_k}\norm{J_k^TF(\bar{x}_k)}+ \frac{\sqrt{L_2^2+\norm{L}^2}}{\sqrt{\gamma}}\norm{x_k-\bar{x}_k}. 
	\end{aligned}
\end{equation}

From Assumptions~\ref{Hip_error_boundnew} and \ref{hipotese5paraconvl}, the choice $\lambda_k = \| J_k^T F_k\|^r$ and \eqref{eq:associationerrorwithL3}, 
we obtain from \eqref{eq:dkGSVDnorm2} that 
\begin{align*}
		\norm{d_k}  & \leq \left( \frac{\kappa L_4}{\gamma \omega^r} + \frac{\kappa C}{\gamma \omega^r} + \frac{\sqrt{L_2^2+\norm{L}^2}}{\sqrt{\gamma}}  \right) \norm{x_k - \bar{x}_k},
\end{align*}
which concludes the proof for (a) with  $c_1 = \kappa(L_4 + C)/\gamma \omega^r + \sqrt{L_2^2+\norm{L}^2}/\sqrt{\gamma}$. 

%(b) If $\lambda_k \geq \underline{\lambda} > 0$, then we obtain
(b) Let us denote 
\begin{equation*}%\label{eq:normamatrizdiag2}
	\Theta :=  \max{\Big[\frac{1}{\lambda_k}, \frac{\kappa}{\lambda_k},\frac{\kappa}{(\sigma^{*}_{\text{min}})^2}, 1 \Big]},
\end{equation*}

Instead of \eqref{eq:dkGSVDnorm2}, we may write 
\begin{equation*}
\| d_k \| \leq  \frac{\Theta}{\gamma}L_4\norm{x_k-\bar{x}_k}^2 +\frac{\Theta}{\gamma}\norm{J_k^T F(\bar{x}_k)} + \frac{\sqrt{L_2^2+\norm{L}^2}}{\sqrt{\gamma}}\norm{x_k-\bar{x}_k}. 
\end{equation*}

Then, from Assumption~\ref{a3}, we have 
\begin{align*}
		\norm{d_k}  & \leq \left( \frac{\Theta}{\gamma } (L_4 \| x_k - \bar{x}_k \| +  \sigma) + \frac{\sqrt{L_2^2+\norm{L}^2}}{\sqrt{\gamma}}  \right) \norm{x_k - \bar{x}_k}.
\end{align*}

Since, $\| x_k - \bar{x}_k \| \leq \delta \leq  \frac{(\sigma_{\min}^*)^2 - \theta \sigma}{\theta L_4}$, it follows that $\lambda_k \leq (\sigma_{\min}^*)^2$. 
Then, $\kappa / \lambda_k \geq \kappa/(\sigma_{\min}^*)^2$. 
Clearly, $\kappa / \lambda_k > 1/ \lambda_k$ because $\kappa > 1$. 
Hence, $\Theta = \max \{ 1, \kappa/\lambda_k \}$. 

Now, assume $\lambda_k$ from \eqref{eq:lambda3}.  In case $\Theta = \kappa / \lambda_k$, the above inequality becomes
\[
		\norm{d_k}  \leq \left( \frac{\kappa}{\gamma }  + \frac{\sqrt{L_2^2+\norm{L}^2}}{\sqrt{\gamma}}  \right) \norm{x_k - \bar{x}_k},
\]
If $\Theta = 1$, from the upper bound on $\delta$ in (b) we obtain
\[
		\norm{d_k}  \leq \left( \frac{(\sigma_{\min}^*)^2}{\theta \gamma }  + \frac{\sqrt{L_2^2+\norm{L}^2}}{\sqrt{\gamma}}  \right) \norm{x_k - \bar{x}_k}.
\]
Therefore, we get the proof for (b) with  $c_1 = (1/\gamma)\max \{ \kappa, (\sigma_{\min}^*)^2/\theta \}  + \sqrt{L_2^2+\norm{L}^2}/\sqrt{\gamma}$. \qed
\end{proof}

\section{Local convergence}\label{sec:local}

Now we focus on the local convergence of the ``pure'' LMMSS iteration, i.e iteration \eqref{s1}--\eqref{s2} with $\alpha_k=1$. We assume the initial point $x_0$ is in a neighborhood of a (possibly nonisolated) stationary point $x^* \in X^*$. 

Lemmas~\ref{lemac1postoconst} and \ref{lemac1postodiminuindo} showed that $\norm{d_k} \leq c_1 \dist(x_k,X^*)$. 
As we shall see in this section, such inequality is key for the local convergence analysis under an error bound condition (Assumption~\ref{Hip_error_boundnew}). 
As we mentioned in the introduction, the analysis developed in this section is based on the work \cite{behling2019local} and extends it to handle singular scaling in LM methods.

Apart from the next lemma which is an intermediate result, the remaining of this section is organized in subsections according to whether the Jacobian rank near the solution set is constant or not. 

%Throughout this section, we consider the LMMSS iteration \eqref{s1}--\eqref{s2} with $\alpha_k = 1$ for every $k$, that is, the ``pure'' LMMSS.

\begin{lema}\label{Lem: Auxiliarconvgeral}
Let $x^* \in X^*$ and $\{x_k\}$ be the sequence generated by the LMMSS (with $\alpha_k=1$). 
Suppose Assumptions~\ref{Hip_Jdiffnew} and \ref{Hip_error_boundnew} holds in $B(x^*,\delta)$, for some $\delta \in ]0,1[$ 
and that the uniform completeness condition is valid in $B(x^*,\delta)\setminus X^*$. 
If $x_{k+1}, x_k \in B(x^*,\delta)$ and $\norm{d_k} \leq c_1\norm{x_k - \bar{x}_k}$, then

\begin{equation}\label{eq:limitacaoomega}
	\begin{aligned}
		\omega \dist(x_{k+1}, X^*) &\leq (L_4c_1^2 + L_5)||x_k - \bar{x}_k||^2 + \lambda_k\norm{L}^2c_1||x_k - \bar{x}_k|| \\&+ ||J(x_k)^TF(\bar{x}_k)|| + ||J(x_{k+1})^TF(\bar{x}_k)||,
	\end{aligned}
\end{equation}

where $L_5 := L_0L_2(1 + c_1)(2 + c_1)$.    
\end{lema}

\begin{proof}
For any $x, y \in \mathbb{R}^n$, the reverse triangle inequality implies
$$
\|\nabla \phi(y) - \nabla \phi(x) - J(x)^TJ(x)(y - x)\| \geq \|\nabla \phi(y)\| - \|\nabla \phi(x) - J(x)^TJ(x)(y - x)\|.
$$

Thus, for $y=x_{k+1}$, $x = x_k$, $x_{k+1}, x_k \in B(x^*,\delta)$ and using Lemma~\ref{lem:L4}, we have: 
\begin{equation}\label{eq:auxilema4.1}
	\begin{aligned}
		\|\nabla \phi(x_{k+1})\| &\leq \|\nabla\phi (x_{k+1}) - \nabla \phi(x_k) - J_k^TJ_kd_k\| + \|\nabla \phi(x_k) + J_k^TJ_kd_k\|\\
		&\leq L_4\|d_k\|^2 + \|(J_k - J_{k+1})^TF_{k+1}\| + \|J_k^TF_k + J_k^TJ_kd_k\|.   
	\end{aligned}
\end{equation}

From the error bound condition (Assumption~\ref{Hip_error_boundnew}), the definition of the LMMSS iteration and \eqref{eq:auxilema4.1}, we get

\begin{equation}\label{eq:auxilema4.1(2)}
	\omega \dist(x_{k+1},X^*) \leq L_4\|d_k\|^2 + \|(J_k - J_{k+1})^TF_{k+1}\| +  \lambda_k\norm{L}^2\norm{d_k}.
\end{equation}

Now, rewriting \eqref{eq:auxhiplips}, we have
\begin{equation}\label{eq:auxilema4.1(3)}
	\begin{aligned}
		\|(J_k - J_{k+1})^TF_{k+1}\| &\leq L_0L_2\|x_k - \bar{x}_k\|\|x_{k+1} - \bar{x}_k\| + \|J_k^TF(\bar{x}_k)\|
		\\&+ L_0L_2\|x_{k+1} - \bar{x}_k\|^2 + \|J_{k+1}^TF(\bar{x}_k)\|.    
	\end{aligned}
\end{equation}

Furthermore, since $\norm{d_k}\leq c_1\norm{x_k-\bar{x}_k}$, it holds  
\begin{equation}\label{eq:auxilema4.1(4)}
	\|x_{k+1} - \bar{x}_k\| = \|x_k - \bar{x}_k + x_{k+1} - x_k\| \leq \|x_k - \bar{x}_k\| + \|d_k\| \leq (1+c_1)\|x_k - \bar{x}_k\|.
\end{equation}

Thus, using \eqref{eq:auxilema4.1(3)} and \eqref{eq:auxilema4.1(4)} in \eqref{eq:auxilema4.1(2)} yields
\begin{equation*}
	\begin{aligned}
		\omega \dist(x_{k+1},X^*) &\leq L_4\|d_k\|^2 + L_0L_2\|x_k - \bar{x}_k\|\|x_{k+1} - \bar{x}_k\| + \|J_k^TF(\bar{x}_k)\|\\
		&+ L_0L_2\|x_{k+1} - \bar{x}_k\|^2 + \|J_{k+1}F(\bar{x}_k)\| +  \lambda_k\norm{L}^2\norm{d_k}\\
		&\leq L_4\|d_k\|^2 +L_0L_2(1+c_1)\|x_k - \bar{x}_k\|^2 + \|J_k^TF(\bar{x}_k)\| \\ &+L_0L_2 (1+c_1)^2\|x_k - \bar{x}_k\|^2 + \|J_{k+1}^TF(\bar{x}_k)\| +  \lambda_k\norm{L}^2\norm{d_k}   \\ &\leq L_4c_1^2\|x_k - \bar{x}_k\|^2 + \lambda_k\norm{L}^2c_1\|x_k - \bar{x}_k\| \\&+ L_0L_2(1 + c_1)\|x_k - \bar{x}_k\|^2 + L_0L_2(1 + c_1)^2\|x_k- \bar{x}_k\|^2 \\&+ \|J_k^TF(\bar{x}_k)\| + \|J_{k+1}^TF(\bar{x}_k)\|,
	\end{aligned}
\end{equation*}
concluding the proof. \qed
\end{proof}

%%%%%%% POSTO CONSTANTE 
\subsection{Constant Rank}\label{subsec:constrank}
In this section, we assume that $1 \leq \text{rank}(J(x^*)) = q \leq \min\{n, m\}$, 
and $\text{rank}(J(x)) = \text{rank}(J(x^*))$ for each $x\in B(x^*, \delta)$. 
From Lemma~\ref{lemac1postoconst}, we recall that $\left\|d_k\right\| \leq c_1 \text{dist}(x_k, X^*)$, 
with $c_1 = \Big(\sqrt{\frac{\norm{L}^2}{\gamma}+1}\Big) \frac{\kappa}{\lambda_p^*}L_3$.
%provided that $\lambda_k > 0,$ for $x_k \neq x^*$.

%%%%%%%%%%%%%%%%%%%%%% HIPOTESE FRACA
Next, we present two lemmas that will aid in proving the convergence theorem under Assumption~\ref{a3}. 

\begin{lema}\label{Lem:Auxconv1hip4}
    Let $x^* \in X^*$, suppose that Assumptions~\ref{Hip_Jdiffnew} and \ref{Hip_error_boundnew} are valid in $B(x^*, \delta)$ for some $\delta \in ]0,1[$, 
    and the uniform completeness condition holds at $B(x^*, \delta)\setminus X^*$. 
    Assume $rank(J(x)) = rank(J(x^*)) \geq 1$ in $B(x^*, \delta)$,  
    that Assumption~\ref{a3} is verified with $\sigma < \bar{\sigma} := \omega/(2 + c_1)$, 
    and $\lambda_k = \|J_k^TF_k\|$.  
    If $x_k, x_{k+1} \in B(x^*, \delta/2)$ and $\dist(x_k, X^*) < \varepsilon$, where
    \begin{equation*}%\label{eq:varepsilonconvhip4}
        \varepsilon \leq\frac{\eta\omega - (2 + c_1)\sigma}{L_6},
    \end{equation*}
    with $\eta \in \left]\frac{\sigma (2 + c_1)}{\omega}, 1 \right[$ and $L_6=L_4c_1^2 + L_5 + L_3\norm{L}^{2}c_1$, then $$\dist(x_{k+1}, X^*) \leq \eta \, \dist(x_k, X^*).$$
\end{lema}
\begin{proof}
Since $J(x)$ has constant rank in $B(x^*, \delta)$, $\lambda_k = \|J_k^TF_k\|$, and Assumption~\ref{a3} is satisfied, from \eqref{eq:limitacaoomega} and \eqref{eqhipL3new}, we have
    \begin{equation*}
        \begin{aligned}
            \omega \dist(x_{k+1}, X^*) &\leq (L_4c_1^2 + L_5)||x_k - \bar{x}_k||^2 + \lambda_k\norm{L}^2c_1||x_k - \bar{x}_k|| \\&+ ||J(x_k)^TF(\bar{x}_k)|| + ||J(x_{k+1})^TF(\bar{x}_k)||  \\
            &\leq (L_4c_1^2 + L_5)\| x_k - \bar{x}_k\|^2 + L_3\norm{L}^2c_1\|x_k - \bar{x}_k\|^2\\
            &+ \sigma\|x_k - \bar{x}_k\| + \sigma\|x_{k+1} - \bar{x}_k\|\\
            &\leq (L_4c_1^2 + L_5 + L_3\norm{L}^2c_1)\|x_k - \bar{x}_k\|^2+ \sigma\|x_k - \bar{x}_k\| \\&+ (1+c_1)\sigma\|x_k - \bar{x}_k\|  \\
            &\leq (L_4c_1^2 + L_5 + L_3\norm{L}^2c_1)\|x_k - \bar{x}_k\|^2 + (2 + c_1)\sigma\|x_k - \bar{x}_k\| \\
            &\leq \Big[(L_4c_1^2 + L_5 + L_3\norm{L}^2c_1)\varepsilon + (2 + c_1)\sigma \Big] \|x_k - \bar{x}_k\|.
        \end{aligned}
    \end{equation*}
    As $\sigma < \frac{\omega}{2 + c_1}$, for any $\eta \in (\frac{\sigma(2 + c_1)}{\omega}, 1)$, we have $\eta\omega - (2 + c_1)\sigma > 0$. 
    Denoting $L_6 := L_4c_1^2 + L_5 + L_3\norm{L}^2c_1$, for $\varepsilon \leq \frac{\eta\omega - (2 + c_1)\sigma}{L_6}$, we obtain $\dist(x_{k+1}, X^*) \leq \eta \dist(x_k, X^*)$.  \qed 
\end{proof}

%The condition $\varepsilon \leq \frac{\frac{\delta}{2}}{ 1 + \frac{c_1}{1 - \eta}} = \frac{\delta}{2} \Big(\frac{1-\eta}{1-\eta+c_1}\Big) < \frac{\delta}{2},$ also allows us to prove the following result.

\begin{lema}
    \cite[Lemma 4.3]{behling2019local}\label{Lem:Auxconv2}
    Suppose that the assumptions of Lemma~\ref{Lem:Auxconv1hip4} are satisfied, and consider 
    \begin{equation}\label{eq:epsA4CR}
            \varepsilon = \min\left\{\frac{\frac{\delta}{2}}{ 1 + \frac{c_1}{1 - \eta}}, \frac{\eta\omega - (2 + c_1)\sigma}{L_6}\right\}.
    \end{equation}
    If $x_0 \in B(x^*, \varepsilon)$, then $x_{k+1} \in B(x^*, \frac{\delta}{2})$ and $\dist(x_k, X^*) \leq \varepsilon$, for every $k \in \mathbb{N}$.
\end{lema}

\begin{proof}
    The proof is done by induction. For $k = 0$, we have $\dist(x_0, X^*) \leq \|x_0 - x^*\| \leq \varepsilon$, thus 
    \[
        \|x_1 - x^*\| \leq \|x_1 - x_0\| + \|x_0 - x^*\| \leq \|d_0\| + \varepsilon
        \leq c_1\dist(x_0, X^*) + \varepsilon \leq (1 + c_1)\varepsilon < \frac{\delta}{2},
    \]
    and by Lemma~~\ref{Lem:Auxconv1hip4}, $ \dist(x_1,X^*)  \leq \eta \dist(x_0,X^*) \leq \varepsilon$.

	From the induction hypothesis for $i \leq k$, $x_i \in B(x^*, \frac{\delta}{2})$ and $\dist(x_{i-1}, X^*) \leq \varepsilon$, for $i = 1, . . . , k$. 
	Then, by successively applying Lemma~\ref{Lem:Auxconv1hip4} to $x_{i-1}$ and $x_i$,  for $i = 1, . . . , k$, we have 
    \[
        \dist(x_k, X^*) \leq \eta \dist(x_{k-1}, X^*) \leq \cdots \leq \eta^k \dist(x_0, X^*) \leq \eta^k \varepsilon < \varepsilon.
    \]	
    
    Finally, using Lemma~\ref{lemac1postoconst},
    \begin{equation*}
        \begin{aligned}
            \|x_{k+1} - x^*\| &\leq \|x_1 - x^*\| + \sum_{i=1}^{k} \|d_i\| \leq \|x_1 - x^*\| + \sum_{i=1}^{k} c_1\dist(x_i, X^*)\\
            &\leq (1 + c_1)\varepsilon + c_1\varepsilon \sum_{i=1}^{\infty} \eta^i = \Big(1+ \frac{ c_1}{1 - \eta}\Big)\varepsilon \leq \frac{\delta}{2},    
        \end{aligned}
    \end{equation*}
    which completes the proof. \qed 
\end{proof}

\begin{teo}\label{teoconvlocahip4}
Let $x^* \in X^*$, suppose that Assumptions~\ref{Hip_Jdiffnew}, \ref{Hip_error_boundnew} and \ref{a3} are valid 
in  $B(x^*, \delta)$ for some $\delta \in ]0,1[$ and $\bar{\sigma} = \omega/(2 + c_1)$,   
$rank(J(x)) = rank(J(x^*)) \geq 1$ in $B(x^*, \delta)$,  
and that the uniform completeness condition holds at $B(x^*, \delta)\setminus X^*$. 
Let $\{ x_k \}$ be generated by the LMMSS method with $\alpha_k=1$, $\lambda_k = \|J_k^TF_k\|$ for all $k$, 
and $x_0 \in B(x^*, \varepsilon)$, where $\varepsilon > 0$ is given by \eqref{eq:epsA4CR}. 
Then, $\{\dist(x_k, X^*)\} $ converges linearly to zero. 
Moreover, the sequence $\{x_k\}$ converges to some $\bar{x} \in X^* \cap B(x^*, \frac{\delta}{2})$.
\end{teo}

\begin{proof}
    The linearly convergence of $\dist(x_k, X^*)$ to zero follows directly from Lemma~\ref{Lem:Auxconv1hip4} and Lemma~\ref{Lem:Auxconv2}.

    Once $\dist(x_k, X^*)$ converges to zero and $x_k \in B(x^*, \frac{\delta}{2})$ for every $k$, it remains to show that $\{ x_k \}$ converges. From Lemma \ref{lemac1postoconst} and Lemma \ref{Lem:Auxconv1hip4}, we have
    \[
        \|d_k\| \leq c_1\text{dist}(x_k, X^*) \leq c_1\eta^k \text{dist}(x_0, X^*) \leq c_1\varepsilon \eta^k,
    \]
    for every $k \geq 1$. Thus, for any positive integers $\ell$ and $q$, with $\ell \geq q$,
    \[
        \|x_\ell - x_q\| \leq \sum_{i=q}^{\ell-1} \|d_i\| \leq \sum_{i=q}^{\infty} \|d_i\| \leq c_1\varepsilon \sum_{i=q}^{\infty} \eta^i \leq  \frac{c_1\varepsilon}{1 - \eta} ,
    \]
    implying that $\{x_k\} \subset \mathbb{R}^n$ is a Cauchy sequence and, therefore, converges. \qed 
\end{proof}

%%%%%%%%%%%%%%%%%%% HIPOTESE FORTE
%Next, we present two lemmas that will aid in proving the convergence theorem under Assumption~\ref{hipotese5paraconvl}.
Now we demonstrate that under Assumption~\ref{hipotese5paraconvl}, which is stronger than Assumption~\ref{a3}, the local convergence can be superlinear. 
In the lemmata and theorems that follows we shall suppress algebraic manipulations that are similar to the results above. 

\begin{lema}\label{Lem:Auxconv1hip5}   
Let $x^* \in X^*$, suppose that Assumptions~\ref{Hip_Jdiffnew} and \ref{Hip_error_boundnew} are valid in  $B(x^*, \delta)$ for some $\delta \in ]0,1[$, 
$rank(J(x)) = rank(J(x^*)) \geq 1$ in $B(x^*, \delta)$,  
and that the uniform completeness condition holds at $B(x^*, \delta)\setminus X^*$. 
Additionally, assume that Assumption~\ref{hipotese5paraconvl} is verified, and $\lambda_k = \|J_k^TF_k\|$. 
If $x_k, x_{k+1} \in B(x^*, \delta/2)$ and $\dist(x_k, X^*) < \varepsilon$, with 
\begin{equation*}%\label{eq:varepsilonconvhip5}
	\varepsilon \leq \left( \frac{\eta \omega}{\tilde{C}} \right)^{1/r},
\end{equation*}
with $\eta \in ]0,1[$, $\tilde{C} = L_6 + (1+(1 + c_1)^{1+r})C$  and $L_6$ as in Lemma~\ref{Lem:Auxconv1hip4}, then 
$$
\dist(x_{k+1}, X^*) \leq \eta \, \dist(x_k, X^*). 
$$
\end{lema}

\begin{proof}
%We know that $J(x)$ has constant rank in $B(x^*, \delta)$, $\lambda_k = \|J_k^TF_k\|$, and Assumption~\ref{hipotese5paraconvl} is satisfied. 
%Thus, from  \eqref{eq:limitacaoomega} and \eqref{eqhipL3new}, we have
Similarly to the proof of Lemma~\ref{Lem:Auxconv1hip4}, we have 
\begin{equation}\label{eq:hip5convlocal}
	\begin{aligned}
		\omega \dist(x_{k+1}, X^*) 
%		& \leq (L_4c_1^2 + L_5)||x_k - \bar{x}_k||^2 + \lambda_k\norm{L}^2c_1||x_k - \bar{x}_k|| \\&+ ||J(x_k)^TF(\bar{x}_k)|| + ||J(x_{k+1})^TF(\bar{x}_k)||  \\
		&\leq (L_4c_1^2 + L_5)\| x_k - \bar{x}_k\|^2 + L_3\norm{L}^2c_1\|x_k - \bar{x}_k\|^2
		\\&+ C\|x_k - \bar{x}_k\|^{1+r} + C\|x_{k+1} - \bar{x}_k\|^{1+r}\\
		&\leq (L_4c_1^2 + L_5 + L_3\norm{L}^2c_1)\|x_k - \bar{x}_k\|^2+ C\|x_k - \bar{x}_k\|^{1+r} \\ &+ C(1 + c_1)^{1+r}\|x_k - \bar{x}_k\|^{1+r}  \\
		&=: L_6\|x_k - \bar{x}_k\|^2 +(1+(1 + c_1)^{1+r})C\|x_k - \bar{x}_k\|^{1+r} \\
		&\leq L_6\|x_k - \bar{x}_k\|^{1+r} +(1+(1 + c_1)^{1+r})C\|x_k - \bar{x}_k\|^{1+r} \\
		&= \Big[L_6 + (1+(1 + c_1)^{1+r})C \Big] \norm{x_k - \bar{x}_k}^{1+r} \\
		& =: \tilde{C} \norm{x_k - \bar{x}_k}^{1+r} = \tilde{C} \dist(x_k,X^*)^{1+r}
	\end{aligned}
\end{equation}
where $L_6 := L_4c_1^2 + L_5 + L_3\norm{L}^2c_1$ and $\tilde{C} := L_6 + (1+(1 + c_1)^{1+r})C$.

From \eqref{eq:hip5convlocal}, we have 
\[
\dist(x_{k+1}, X^*) \leq \dfrac{\tilde{C}}{\omega}\dist(x_k,X^*)^{1+r} \leq \dfrac{\tilde{C}}{\omega}\varepsilon^r \dist(x_k,X^*)
\]
and, given $\eta \in ]0,1[$, for $\varepsilon \leq \left( \frac{\eta \omega}{\tilde{C}} \right)^{1/r}$ we obtain
\[
\dist(x_{k+1}, X^*) \leq \eta \dist(x_k,X^*),
\]
which concludes the proof. \qed 
\end{proof}

\begin{lema}\label{Lem:Aux2conv1hip5}
Suppose that the assumptions of Lemma~\ref{Lem:Auxconv1hip5} are satisfied, and  
\begin{equation}\label{eq:epsA5CR}
	\varepsilon = \min \left\{ \frac{\frac{\delta}{2}}{ 1 + \frac{c_1}{1 - \eta}} , \left( \frac{\eta \omega}{\tilde{C}} \right)^{1/r} \right\}.    
\end{equation}
If $x_0 \in B(x^*, \varepsilon)$, then $x_{k+1} \in B(x^*, \frac{\delta}{2})$ and $\dist(x_k, X^*) \leq \varepsilon$, for every $k \in \mathbb{N}$. 
\end{lema}
\begin{proof}
Same as that of Lemma~\ref{Lem:Auxconv2}. \qed 
%%%%The proof is done by induction. For $k = 0$, we have $\dist(x_0, X^*) \leq \|x_0 - x^*\| \leq \varepsilon$ and
%%%%\[
%%%%\|x_1 - x^*\| \leq \|x_1 - x_0\| + \|x_0 - x^*\| \leq \|d_0\| + \varepsilon
%%%%\leq c_1\dist(x_0, X^*) + \varepsilon \leq (1 + c_1)\varepsilon < \frac{\delta}{2}.
%%%%\]
%%%%
%%%%For $k \geq 1$, assume that $x_i \in B(x^*, \frac{\delta}{2})$ and $\dist(x_{i-1}, X^*) \leq \varepsilon$ for $i = 1, . . . , k$. Then, from Lemma~\ref{Lem:Auxconv1hip5} applied to $x_{k-1}$ and $x_k$, we obtain
%%%%\[
%%%%\|x_k - \bar{x}_k\| = \dist(x_k, X^*) \leq \eta \dist(x_{k-1}, X^*) \leq \eta \varepsilon < \varepsilon.
%%%%\]
%%%%
%%%%Moreover, by Lemma~\ref{Lem:Auxconv1hip5} and the induction hypothesis for $i \leq k$, we have
%%%%\[
%%%%\dist(x_i, X^*) \leq \eta \dist(x_{i-1}, X^*) \leq \cdots \leq \eta^i \dist(x_0, X^*) \leq \eta^i \varepsilon < \varepsilon.
%%%%\]
%%%%Thus, using Lemma~\ref{lemac1postoconst},
%%%%\begin{equation}
%%%%	\begin{aligned}
%%%%		\|x_{k+1} - x^*\| &\leq \|x_1 - x^*\| + \sum_{i=1}^{k} \|d_i\| \\
%%%%		&\leq \|x_1 - x^*\| + \sum_{i=1}^{k} c_1\dist(x_i, X^*)\\
%%%%		&\leq (1 + c_1)\varepsilon + c_1\varepsilon \sum_{i=1}^{\infty} \eta^i \\
%%%%		&= \Big(1+ \frac{ c_1}{1 - \eta}\Big)\varepsilon \leq \frac{\delta}{2},     
%%%%	\end{aligned}
%%%%\end{equation}
%%%%concluding the proof. \qed 
\end{proof}

\begin{teo}\label{teoconvlocalhip5CR}
Let $x^* \in X^*$, suppose that Assumptions~\ref{Hip_Jdiffnew}, \ref{Hip_error_boundnew} and \ref{hipotese5paraconvl} are valid 
in  $B(x^*, \delta)$ for some $\delta \in ]0,1[$, 
$rank(J(x)) = rank(J(x^*)) \geq 1$ in $B(x^*, \delta)$,  
and that the uniform completeness condition holds at $B(x^*, \delta)\setminus X^*$. 
Let $\{ x_k \}$ be generated by the LMMSS method with $\alpha_k=1$, $\lambda_k = \|J_k^TF_k\|$ for all $k$, 
and $x_0 \in B(x^*, \varepsilon)$, where $\varepsilon > 0$ is given by \eqref{eq:epsA5CR}. 
Then, $\{\dist(x_k, X^*)\} $ converges superlinearly to zero. 
Moreover, the sequence $\{x_k\}$ converges to some $\bar{x} \in X^* \cap B(x^*, \frac{\delta}{2})$.
\end{teo}
\begin{proof}
From  \eqref{eq:hip5convlocal} we obtain
\begin{equation*}
	\dfrac{\dist(x_{k+1},X^*)}{\dist(x_k,X^*)} \leq \dfrac{\tilde{C}}{\omega}\dist(x_k,X^*)^r
\end{equation*}
and since $\dist(x_k,X^*)^r \rightarrow 0$ (due to Lemmas~\ref{Lem:Auxconv1hip5}  and \ref{Lem:Aux2conv1hip5}) we have that $\{\dist(x_k, X^*)\}$ converges to zero superlinearly.

Proof of convergence of the whole sequence $\{x_k\}$ to some $\bar{x} \in X^* \cap B(x^*, \frac{\delta}{2})$ follows the same lines as in the proof of Theorem~\ref{teoconvlocahip4}. \qed 

%%%Once $\dist(x_k, X^*)$ converges to zero and $x_k \in B(x^*, \frac{\delta}{2})$ for every $k$, it remains to show that $\{ x_k \}$ converges. 
%%%From Lemma \ref{lemac1postoconst} and Lemma \ref{Lem:Auxconv1hip5}, we have
%%%\[
%%%\|d_k\| \leq c_1\text{dist}(x_k, X^*) \leq c_1\eta^k \text{dist}(x_0, X^*) \leq c_1\varepsilon \eta^k,
%%%\]
%%%for every $k \geq 1$. Thus, for any positive integers $\ell$ and $q$, with $\ell \geq q$,
%%%\[
%%%\|x_\ell - x_q\| \leq \sum_{i=q}^{\ell-1} \|d_i\| \leq \sum_{i=q}^{\infty} \|d_i\| \leq c_1\varepsilon \sum_{i=q}^{\infty} \eta^i,
%%%\]
%%%implying that $\{x_k\} \subset \mathbb{R}^n$ is a Cauchy sequence and, therefore, converges. \qed
\end{proof}

%%\begin{obs}
%%In particular, when $r=1$ in Assumption~\ref{hipotese5paraconvl}, $\dist(x_k,X^*)$ converges to zero quadratically. 
%%\end{obs}

\subsection{Diminishing Rank}\label{subsecdimirank}

In this section, we will examine the scenario where the rank of $J(x_k)$ decreases as $x_k$ approaches the set of stationary points $X^*$. 
In this case, the convergence analysis will depend not only on the upper bound for $\| J(x_k)^T F(\bar{x}_k) \|$ given in 
Assumption~\ref{a3} (or Assumption~\ref{hipotese5paraconvl}) but also on the specific choices of the LM parameter $\lambda_k$.

	\begin{lema}\label{Lem:Auxconvpostodimhip5}
Let $x^* \in X^*$, suppose that Assumptions~\ref{Hip_Jdiffnew} and \ref{Hip_error_boundnew} are valid in  $B(x^*, \delta)$ for some $\delta \in ]0,1[$, 
and that the uniform completeness condition holds at $B(x^*, \delta)\setminus X^*$. 
Additionally, assume that Assumption~\ref{hipotese5paraconvl} is verified with $r \in ]0,1]$, and $\lambda_k = \|J_k^TF_k\|^r$. 
		If $x_k, x_{k+1} \in B(x^*, \delta/2)$ and $\dist(x_k, X^*) < \varepsilon$, where
		\begin{equation*}%\label{eq:varepsilonconvhip5dimin}
			\varepsilon \leq \Big(\frac{\eta\omega }{\hat{C}}\Big)^{1/r}
		\end{equation*}
		with $\eta \in (0, 1)$, and $\hat{C} = L_7+(1+(1 + c_1)^{1+r})C$, $L_7 =L_4c_1^2 + L_5 + L_3^r\norm{L}^2c_1$, 
		then 
		$$
		\dist(x_{k+1}, X^* ) \leq \eta \, \dist(x_k, X^*). 
		$$
	\end{lema}
	
	\begin{proof}
		Note that from Assumption~\ref{hipotese5paraconvl} and Lemma~\ref{lemac1postodiminuindo}(a), for $\lambda_k= \norm{J_k^TF_k}^r$, we have $\norm{d_k}\leq c_1\dist(x_k,X^{*})$, 
		with 
		\begin{equation}\label{c1compostodiminuindohip5}
			c_1=\frac{4}{\gamma \omega^r}(L_4+C)+ \frac{\sqrt{L_2^2+\norm{L}^2}}{\sqrt{\gamma}}.
		\end{equation}
		
		From Lemma~\ref{Lem: Auxiliarconvgeral}, Assumptions~\ref{Hip_error_boundnew} and \ref{hipotese5paraconvl}, $\lambda_k=\norm{J_k^TF_k}^r$,  $\dist(x_k,X^{*}) \leq \delta/2<1/2$ and assuming that $\dist(x_k, X^{*})< \varepsilon$, we have
		\begin{equation*}
			\begin{aligned}
				\omega \dist(x_{k+1}, X^*) 
%				&\leq (L_4c_1^2 + L_5)||x_k - \bar{x}_k||^2 + \lambda_k\norm{L}^2c_1||x_k - \bar{x}_k|| + \\
%				&+ ||J(x_k)^TF(\bar{x}_k)|| + ||J(x_{k+1})^TF(\bar{x}_k)||  \\
				&\leq (L_4c_1^2 + L_5)\| x_k - \bar{x}_k\|^2 + L_3^{r}\norm{L}^2c_1\|x_k - \bar{x}_k\|^{1+r} \\ &+ C\|x_k - \bar{x}_k\|^{1+r} + C\|x_{k+1} - \bar{x}_k\|^{1+r}\\
				&\leq (L_4c_1^2 + L_5 + L_3^r\norm{L}^2c_1)\|x_k - \bar{x}_k\|^2 \\&+ (1+(1 + c_1)^{1+r})C\|x_k - \bar{x}_k\|^{1+r} \\
				&= [L_7+(1+(1 + c_1)^{1+r})C]\|x_k - \bar{x}_k\|^{1+r} \\
				&\leq [L_7+(1+(1 + c_1)^{1+r})C]\varepsilon^{r} \norm{x_k - \bar{x}_k}\\
				&:=\hat{C}\varepsilon^{r}\norm{x_k - \bar{x}_k},
			\end{aligned}
		\end{equation*}
		
		with $c_1$ as in \eqref{c1compostodiminuindohip5} and $L_7= L_4c_1^2 + L_5 + L_3^r\norm{L}^2c_1$.
		
		Therefore, for $\varepsilon\leq \Big(\frac{\eta\omega }{\hat{C}}\Big)^{1/r},$ we have $\dist(x_{k+1}, X^*) \leq \eta \, \dist(x_k, X^*)$. \qed 
	\end{proof}
	
	Using an appropriate $\varepsilon > 0$, and Lemma~\ref{Lem:Auxconvpostodimhip5} , the proofs of the next results are analogous to those of Lemma~\ref{Lem:Aux2conv1hip5} and Theorem~\ref{teoconvlocalhip5CR}, respectively.
	
	\begin{lema}\label{Lem:Auxconv2diminuindohip5}
		Suppose that the Assumptions of Lemma~\ref{Lem:Auxconvpostodimhip5} are satisfied, and consider $\varepsilon$ as
		\begin{equation}\label{eq:epsA5DIMRANK}
			\varepsilon = \min\left\{\frac{\frac{\delta}{2}}{ 1 + \frac{c_1}{1 - \eta}}, \Big(\frac{\eta\omega }{\hat{C}}\Big)^{1/r}\right\}.
		\end{equation} 
		If $x_0 \in B(x^*, \varepsilon)$, then $x_{k+1} \in B(x^*, \frac{\delta}{2})$, and $\dist(x_k, X^*) \leq \varepsilon$ for every $k \in \mathbb{N}$. 
	\end{lema}
	
	\begin{teo}\label{teoconvlocalhip5rankdim}
Let $x^* \in X^*$, suppose that Assumptions~\ref{Hip_Jdiffnew}, \ref{Hip_error_boundnew} and \ref{hipotese5paraconvl} are valid 
in  $B(x^*, \delta)$ for some $\delta \in ]0,1[$ and $r \in ]0,1]$,    
and that the uniform completeness condition holds at $B(x^*, \delta)\setminus X^*$. 
Let $\{ x_k \}$ be generated by the LMMSS method with $\alpha_k=1$, $\lambda_k = \|J_k^TF_k\|^r$ for all $k$, 
and $x_0 \in B(x^*, \varepsilon)$, where $\varepsilon > 0$ is given by \eqref{eq:epsA5DIMRANK}. 
Then, $\{\dist(x_k, X^*)\} $ converges superlinearly to zero. 
Moreover, the sequence $\{x_k\}$ converges to some $\bar{x} \in X^* \cap B(x^*, \frac{\delta}{2})$.	
	\end{teo}
	
	\begin{obs}
		The main difference from Theorem~\ref{teoconvlocalhip5CR} is that when Assumption~\ref{hipotese5paraconvl} holds with $0 < r < 1$ 
		we cannot choose $\lambda_k = \| J_k^T F_k \|$ anymore;  $\lambda_k = \| J_k^T F_k \|^r$ is essential for the convergence proof in this case. 
		This becomes clear from the proof of Lemma~\ref{lemac1postodiminuindo}(a); see inequality~\eqref{eq:dkGSVDnorm2}. In fact, we could choose $\lambda_k=\| J_k^T F_k \|^{r'}$ with $r'\leq r$ that all analysis remains valid with minor modifications.
	\end{obs}

	%%%%%%%%% POSTO DECRESCENTE E HIPOTESE FRACA
To close this section, we analyze local convergence of LMMSS when the Jacobian rank is not constant and Assumption~\ref{a3} holds. 
	
	\begin{lema}\label{lem:Auxconv1a3}
Let $x^* \in X^*$, suppose that Assumptions~\ref{Hip_Jdiffnew} and \ref{Hip_error_boundnew} are valid in  $B(x^*, \delta)$ for some $\delta \in ]0,1[$ 
and that the uniform completeness condition holds at $B(x^*, \delta)\setminus X^*$. 	
    Also, assume that Assumption~\ref{a3} is verified with 
    \[
    \sigma < \bar{\sigma} := \min \left\{ (\sigma^*_{\min})^2, \frac{\omega}{(2 + (1 + \theta \|L\|^2))c_1} \right\},
    \]
    and $\lambda_k$, $\theta$ and $\bar{\delta}$ are those in Lemma~\ref{lemac1postodiminuindo}(b).
    If $x_k, x_{k+1} \in B(x^*, \bar{\delta}/2)$, 
    and $\dist(x_k, X^*) < \varepsilon$, where
    \begin{equation*}%\label{eq:varepsilonconva3}
        \varepsilon \leq \frac{\eta\omega - (2 + (1 + \theta \|L\|^2) c_1)\sigma}{L_8}, 
    \end{equation*}
    with $\eta \in \left] \frac{ (2 + (1 + \theta \|L\|^2) c_1)\sigma}{\omega}, 1 \right[$ and $L_8=L_4(c_1^2 + \theta c_1 \| L \|^2) + L_5$, 
    then 
    \[
    \dist(x_{k+1}, X^*) \leq \eta \, \dist(x_k, X^*).
    \]
\end{lema}
\begin{proof}
%From Assumption~\ref{a3}, \eqref{eq:limitacaoomega} and \eqref{eqhipL3new}, we have

		From Assumption~\ref{a3} and Lemma~\ref{lemac1postodiminuindo}(b), for $\lambda_k$ as in \eqref{eq:lambda3},  
		we have $\norm{d_k}\leq c_1\dist(x_k,X^{*})$, 
		with 
		\begin{equation}\label{c1compostodiminuindoa3}
 		c_1 = (1/\gamma)\max \{ \kappa, (\sigma_{\min}^*)^2/\theta \}  + \sqrt{L_2^2+\norm{L}^2}/\sqrt{\gamma}.
%			c_1 = \Theta(L_4 + \sigma)/\gamma  + \sqrt{L_2^2+\norm{L}^2}/\sqrt{\gamma}.
		\end{equation}
		
		From Lemma~\ref{Lem: Auxiliarconvgeral}, Assumptions~\ref{Hip_error_boundnew} and \ref{a3},  $\dist(x_k,X^{*}) \leq \delta/2<1/2$ and assuming that $\dist(x_k, X^{*})< \varepsilon$, we have
    \begin{equation*}
        \begin{aligned}
            \omega \dist(x_{k+1}, X^*) 
%            &\leq (L_4c_1^2 + L_5)||x_k - \bar{x}_k||^2 + \lambda_k \norm{L}^2c_1||x_k - \bar{x}_k|| \\&+ ||J(x_k)^TF(\bar{x}_k)|| + ||J(x_{k+1})^TF(\bar{x}_k)||  \\
            &\leq (L_4c_1^2 + L_5)\| x_k - \bar{x}_k\|^2 + \lambda_k \norm{L}^2c_1||x_k - \bar{x}_k||\\
            &+ \sigma\|x_k - \bar{x}_k\| + \sigma\|x_{k+1} - \bar{x}_k\|\\
            &\leq (L_4(c_1^2 + \theta c_1 \| L \|^2) + L_5 )\|x_k - \bar{x}_k\|^2 +  (2+ (1 + \theta \|L\|^2)c_1)\sigma\|x_k - \bar{x}_k\|  \\
			&\leq \Big[ (L_4(c_1^2 + \theta c_1 \| L \|^2) + L_5 )\varepsilon  +  (2+ (1 + \theta \|L\|^2)c_1)\sigma  \Big] \|x_k - \bar{x}_k\|.
        \end{aligned}
    \end{equation*}
    Now, for $\sigma < \frac{\omega}{(2 + (1 + \theta \|L\|^2))c_1}$, for any $\eta \in \left] \frac{(2 + (1 + \theta \|L\|^2))\sigma}{\omega}, 1 \right[$, 
    we have $\eta\omega - (2 + (1 + \theta \|L\|^2) c_1)\sigma > 0$. 
    Denoting $L_8 = L_4(c_1^2 + \theta c_1 \| L \|^2) + L_5$, for $\varepsilon \leq \frac{\eta\omega - (2 + (1 + \theta \|L\|^2) c_1)\sigma}{L_8}$, 
    we obtain $\dist(x_{k+1}, X^*) \leq \eta \dist(x_k, X^*)$.  \qed 
\end{proof}

	Again, with the aid of Lemma~\ref{lem:Auxconv1a3}, the proofs of the next results are analogous to those 
	of Lemma~\ref{Lem:Auxconvpostodimhip5} and Theorem~\ref{teoconvlocalhip5rankdim}, respectively.
	
	\begin{lema}\label{Lem:Auxconv2diminuindoA3}
		Suppose that the Assumptions of Lemma~\ref{lem:Auxconv1a3} are satisfied in $B(x^*, \delta)$, 
		consider $\bar{\delta}$ from Lemma~\ref{lemac1postodiminuindo}(b) and 
		\begin{equation}\label{eq:epsA3DIMRANK}
			\varepsilon = \min\left\{\frac{\frac{\bar{\delta}}{2}}{ 1 + \frac{c_1}{1 - \eta}}, \frac{\eta\omega - (2 + (1 + \theta \|L\|^2) c_1)\sigma}{L_8} \right\}.
		\end{equation} 
		If $x_0 \in B(x^*, \varepsilon)$, then $x_{k+1} \in B(x^*, \frac{\bar{\delta}}{2})$, and $\dist(x_k, X^*) \leq \varepsilon$ for every $k \in \mathbb{N}$. 
	\end{lema}
	
	\begin{teo}\label{teoconvlocalA3rankdim}
Let $x^* \in X^*$, suppose that the assumptions of Lemma~\ref{lem:Auxconv1a3} hold in $B(x^*, \delta)$. 
		Let $\{x_k\}$ be generated by LMMSS with $\alpha_k=1$, 
		$\lambda_k$ as in \eqref{eq:lambda3}, for all $k$, $\theta, \bar{\delta}$ from Lemma~\ref{lemac1postodiminuindo}(b). 
		If $x_0 \in B(x^*, \varepsilon)$, where $\varepsilon > 0$ as in \eqref{eq:epsA3DIMRANK},
		 then $\{dist(x_k,X^{*})\}$ converges linearly to zero. Furthermore, the sequence $\{x_k\}$ converges to the solution $\bar{x} \in X^{*}\cap  B(x^*, \bar{\delta}/2)$.
	\end{teo}
	
	\begin{obs}
		In the results under Assumption~\ref{a3}  that led to Theorem~\ref{teoconvlocalA3rankdim}, 
		we considered that the Jacobian rank may change in a neighborhood of $X^*$. 
		In this case, we emphasize that, in order to apply Lemma~\ref{lemac1postodiminuindo}(b), 
		$\lambda_k$ must follow the regime defined by \eqref{eq:lambda3}, 
		which implies $\sigma \leq \lambda_k \leq \theta L_4 \delta + \theta \sigma$. 
		Thus, in this setting, $\lambda_k$ must stays bounded away from zero, 
		as opposed to the previous results in this work. 
	\end{obs}
	
	%%%%%%%%	%%%%%%%%	%%%%%%%%	%%%%%%%%	%%%%%%%%	%%%%%%%%

	\section{Global convergence}\label{sec:global}
	Now we consider Algorithm~\ref{alg:LMM_armijo}, a version of LMMSS globalized with line-search for nonzero residual nonlinear least-squares. 
	It is worth to point out that Algorithm~\ref{alg:LMM_armijo} differs from the algorithm proposed in \cite{Boos2024} in:
	\begin{itemize}
	\item the choice of the LM parameter $\lambda_k$,
	\item  the full step ($\alpha_k=1$) acceptance criterion, 
	\item a LMM direction as safeguard when LMMSS direction fails to satisfy sufficient descent conditions. 
	\end{itemize}

	\LinesNumberedHidden{
		\begin{algorithm}
			\caption{Globalized LMMSS}\label{alg:LMM_armijo}
			\textbf{Input:} $\nu, \zeta, \vartheta, \xi \in ]0,1[$, $M > 0$, $F$, $J$, $L$, $\lambda_0 > 0$ and $x_0 \in \Rn$. \\
			Set $k=0$. 
			\begin{enumerate}
				\item If $\| \nabla \phi(x_k) \| = 0$, stop and return $\bar{x} = x_{k}$. 
				\item Choose $\lambda_k>0$ and compute $d_k$, a solution of  $(J_k^T J_k + \lambda_k L^T L) d = - J_k^T F_k$. 
				\item If $\norm{\nabla \phi (x_k+ d_k)} \leq \vartheta \norm{\nabla \phi (x_k)}$, set $\alpha_k=1$ and go to Step~6. 
				\item If $\|d_k\| > M$ or $-\nabla \phi(x_k)^T d_k  < \xi \| \nabla \phi(x_k)\|^2$, 
				set 
				$$
				d_k = -(J_k^T J_k + \lambda_k I)^{-1} J_k^T F_k.
				$$
				\item 			Choose $m$ as the smallest nonnegative integer such that 
				\begin{equation*}%\label{armijo}
					\phi (x_k+ \zeta ^m{d}_k) - \phi ({x}_k) \leq \nu \zeta^m \nabla \phi ({x}_k)^T {d}_k
				\end{equation*} 
				Set $\alpha_k = \zeta^m$. 
				\item Update ${x}_{k+1} = {x}_k+ \alpha_k {d}_k$,  $k \leftarrow k+1$, \\and go to Step 1.
			\end{enumerate}
	\end{algorithm}}
	
		As stated, Algorithm~\ref{alg:LMM_armijo} leaves open the choice of the sequence of positive LM parameters $\{ \lambda_k \}$. 
		Except for the very last theorem of this section, all the results that follow hold for the choices $\lambda_k = \|J_k^T F_k\|^r, r \in ]0,1]$ 
		or $\lambda_k \in [ \underline{\lambda}, \bar{\lambda} ]$,  $0 < \underline{\lambda}  \leq \bar{\lambda} < +\infty$. %where $\underline{\lambda}, \bar{\lambda}$ are not necessarily those of the Section~\ref{sec:local}. 
		
		We remark that the linear system in Step~2 always has a solution, independently whether condition~\eqref{comp-condition} holds at $x_k$. 
		This is because such linear system is of the form $B^T B d = B^T c$, where $B^T = [J_k^T \ \sqrt{\lambda_k}L^T ]$ and $c^T = [F_k \ 0 ]$.
		
		Initially, we will demonstrate that the direction sequence $\{d_k\}$ generated by Algorithm~\ref{alg:LMM_armijo} is gradient-related to $\{x_k\}$ (see \cite[Eq(1.14)]{Bertsekas2016}). Subsequently, we will establish that any limit point of the sequence produced by this algorithm is a stationary point for \eqref{prob1}, regardless of the initial point. 
		
%%%		In this section, we assume that Assumption~\ref{hipotese5paraconvl} is satisfied with $r=1$ to simplify notation and proofs. 

		We recall the definition of gradient-related directions from \cite[Eq(1.14)]{Bertsekas2016}.
		
		\begin{defi}[Gradient-related]\label{def_grad_related}
			Let $\{x_k \}$ and $\{d_k\}$ be sequences in $\R^n$ and $\phi: \Rn \to \R$ a continuously differentiable function. 
			The sequence $\{d_k\}$ is said to be gradient-related to $\{ x_k \}$ if, for each subsequence $\{x_k\}_{k\in \mathcal{K}}$ (with $\mathcal{K} \subseteq \mathbb{N}$) converging to a nonstationary point of $\phi$, the corresponding subsequence $\{d_k\}_{k\in \mathcal{K}}$ is bounded and satisfies 
			\begin{equation*}%\label{liminf_dk}
				\limsup_{k \rightarrow \infty, \ k\in \mathcal{K}} \nabla \phi (x_k)^T d_k < 0.
			\end{equation*}
		\end{defi}
		
				\begin{obs}\label{obs:sufficient}  
				Let $\{d_k\}$, $\{x_k\}$ be sequences generated by an iterative algorithm $x_{k+1} = x_k + \alpha_k d_k$, 
				for minimizing $\phi(x)$, where $\alpha_k$ fulfills Armijo condition. 
				It is not hard to prove that if  $\{d_k\}$, $\{x_k\}$ satisfy 
				\begin{equation*}%\label{eq:cond-suf}
				\| d_k \| \leq \bar{M} \quad \text{and} \quad \nabla \phi(x_k)^T d_k \leq -\bar{\xi} \| \nabla \phi(x_k) \|^{p_1}, 
				\end{equation*}
				for $\bar{M} > 0$, $p_1 > 0$ and $\bar{\xi} \in ]0,1[$, then $\{d_k\}$ is gradient-related. 
				\end{obs}
		
		\begin{prop}\label{prop_dk}  
			Let $\{d_k\}$, $\{x_k\}$ be sequences generated by Algorithm~\ref{alg:LMM_armijo} starting from $x_0 \in \Rn$. 
			Suppose that condition~\eqref{comp-condition} 	holds in the level set ${\cal L}_0 = \{ x \in \Rn \mid \phi(x) \leq \phi(x_0)\}$, 
			and that either $\lambda_k = \| J_k^T F_k \|^r$ or $0 < \underline{\lambda} \leq \lambda_k \leq \bar{\lambda}$. 
			Then, $\{d_k\}$ is \textup{gradient-related}.
		\end{prop}
		
		\begin{proof}
			Suppose that Algorithm~\ref{alg:LMM_armijo} generates an infinite sequence $\{x_k\}$, i.e., $\| \nabla \phi(x_k) \| = \| J_k^T F_k \| > 0$,  for each $k$. 
			Then, $\norm{F(x_k)}>0$ in this sequence. 
			Let $\{x_k\}_{k \in \mathcal{K}}$, $\mathcal{K} \subseteq \mathbb{N}$, be a subsequence converging to $\hat{x}$, 
			a nonstationary point for $\phi$, i.e., $\nabla \phi (\hat{x}) \neq 0$.
			From now on, consider $k\in \mathcal{K}$. 
			From \eqref{JmuL} and the definition of $d_k$, it follows that
			\begin{align}
				\nabla \phi (x_k)^T d_k & = - d_k^T(J_k^T J_k + \lambda_k L^T L) d_k \nonumber \\
				& = -d_k^T X_k^{-T} \left[ \begin{array}{cc}
					\Sigma_k^2 + \lambda_k M_k^2 & 0 \nonumber \\
					0 & I_{n-p} \end{array}  \right] X_k^{-1} d_k \nonumber \\
				& :=  -d_k^T X_k^{-T} \Lambda_k X_k^{-1} d_k. \label{eq:gtd1}
			\end{align}
			
			By Ostrowski's inertia law \cite[Theorem~1]{ostrowski}, we have
			\[
			\lambda_n ( X_k^{-T} \Lambda_k X_k^{-1} ) = \sigma_n^2(X_k^{-1}) \lambda_n (\Lambda_k),
			\]
			where, with an abuse of notation, $\lambda_n(B)$ and $\sigma_n(B)$ represent the smallest eigenvalue and the smallest singular value of a matrix $B$, respectively.
			Then, from Lemma~\ref{lem:normXk} it follows that
			\begin{equation}\label{eq:gtd2}
			\lambda_n ( X_k^{-T} \Lambda_k X_k^{-1} ) \geq \gamma \lambda_n (\Lambda_k) \geq \gamma \min \{ 1, \lambda_k \},			
			\end{equation}
			where, in the last inequality, we use the fact that $\sigma_{i,k}^2 + \mu_{i,k}^2 = 1$.
%			Since $\hat{x}$ is nonstationary, $\norm{J_k^TF_k}>0$, for $k \in {\cal K}$. 
			
			If $\lambda_k$ is chosen such that $\lambda_k \geq \underline{\lambda} > 0$, from \eqref{eq:gtd1} and \eqref{eq:gtd2} we obtain
			\begin{equation}\label{eq:angle}
				\nabla \phi (x_k)^T d_k \leq - \gamma \min \{1,\underline{\lambda} \} \|d_k\|^2. 
			\end{equation}

			If $\lambda_k= \norm{J_k^TF_k}^r$, \eqref{eq:angle} also holds, because $\norm{J_k^TF_k}>0$, for $k \in {\cal K}$, 
			and $\hat{x}$ is nonstationary, which implies the existence of 
			 $\underline{\lambda} > 0$ such that for $k \in {\cal K}$, $\lambda_k = \| J_k^T F_k \| \geq \underline{\lambda}$.
			
			On the other hand,
			\begin{align*}
				\| \nabla \phi(x_k) \| & = \| (J_k^T J_k + \lambda_k L^T L) d_k \| \leq \| J_k^T J_k + \lambda_k L^T L \| \| d_k \| \\
				& \leq \left( \|J_k\|^2 + \lambda_k\| L \|^2 \right) \| d_k \|. \\
			\end{align*}
			
			By the continuity of $F(x)$ and $J(x)$ and the convergence of the subsequence $\{x_k\}_\mathcal{K}$, there exist $L_J > 0$ and $M_0>0$ such that $\|J(x_k)^TF(x_k)\| \leq M_0$ and $\norm{J(x_k)} \leq L_J$ for each $k \in {\cal K}$. 
			
			Thus, if either $0 < \underline{\lambda} \leq \lambda_k \leq \bar{\lambda}$ or $\lambda_k=\norm{J_k^TF_k}^r \leq M_0^r =: \bar{\lambda}$, we have 
			\begin{equation}\label{eq:prop}
				\| \nabla \phi(x_k) \| \leq \left( L_J^2 + \bar{\lambda} \| L \|^2 \right) \| d_k \| =: M_1\|d_k\|, \quad \forall k \in {\cal K}.
			\end{equation}
			
			From \eqref{eq:prop} and \eqref{eq:angle}, we obtain
			\[
			\nabla \phi (x_k)^T d_k \leq - \frac{\gamma \min \{1,\underline{\lambda} \}}{M_1^2} \| \nabla \phi (x_k) \|^2, \quad \forall k \in {\cal K},
			\]
			which implies that
			\[
			\limsup_{k \rightarrow \infty, \ k\in \mathcal{K}} \nabla \phi (x_k)^T d_k \leq - \frac{\gamma \min \{1,\underline{\lambda} \}}{M_1^2} \| \nabla \phi (\hat{x}) \|^2 < 0.
			\]
			
			Now we just need to show that $\{d_k\}_{k\in \mathcal{K}}$ is bounded. 

			Recall from \eqref{eq:dkGSVD} that 
			\begin{equation*}%\label{est_prop_dk}
				\|d_k\| \leq \|X_k\|^2 \max \{ \|\Gamma_k \|, 1 \} \|J_k^TF_k\|, \quad \forall k\geq 0,
			\end{equation*}
			with $\Gamma_k = (\Sigma_k^2 + \lambda_k M_k^2)^{-1}$.
			
			From Lemma~\ref{lem:normXk}, we have $\|X_k\| \leq 1/\sqrt{\gamma}$. Now let us analyze the term $\max \{ \|\Gamma_k \|, 1 \} \| J_k^T F_k \| $.
			
			From Lemma~\ref{lemma_psi}, it follows that
			\begin{equation*}%\label{norm_Gamma_k}
				\|\Gamma_k\| \leq \dfrac{1}{ \lambda_k}, \quad \text{if} \quad 0 < \lambda_k < 1, \quad \text{and} \quad \|\Gamma_k\| \leq 1, \quad \text{if} \quad \lambda_k \geq 1.
			\end{equation*}
			
			We separate into two cases to analyze such upper bound. 
			\begin{itemize}
				\item[(a)] If $0 < \lambda_k < 1$, then
				\begin{equation*}
					\max \{ \|\Gamma_k \|, 1 \} \leq \dfrac{1}{ \lambda_k}.
				\end{equation*}
				If $\lambda_k = \|J_k^TF_k\|^r$, we have 
				%But we already seen that for $k \in {\cal K}$, there exist $\underline{\lambda}>0$ such that $\lambda_k \geq \underline{\lambda}$, thus   
				\[
					\max \{ \|\Gamma_k \|, 1 \} \|J_k^TF_k\| \leq \dfrac{1}{ \lambda_k} \|J_k^TF_k\| = \dfrac{1}{ \|J_k^TF_k\|^r} \|J_k^TF_k\|=\|J_k^TF_k\|^{1-r}.
				\]
				If $\lambda_k \geq \underline{\lambda} > 0$, we obtain 
				\[
					\max \{ \|\Gamma_k \|, 1 \} \|J_k^TF_k\| \leq \dfrac{1}{ \lambda_k} \|J_k^TF_k\| = \dfrac{1}{ \underline{\lambda}} \|J_k^TF_k\|.
				\]			
				In both cases 	the right hand side is bounded by $\max \left\{ M_0^{1-r}, M_0/\underline{\lambda} \right\}$.
				
				\item[(b)] If $\lambda_k \geq 1$, then $\max \{ \|\Gamma_k \|, 1 \} \leq 1$. Thus,
				\begin{equation*}
					\max \{ \|\Gamma_k \|, 1 \} \|J_k^TF_k\| \leq \|J_k^TF_k\| \leq M_0, \quad \forall k \in {\cal K}.
				\end{equation*}
			\end{itemize}
			Therefore, from (a) and (b), we conclude that 
			\begin{equation*}
				\max \{ \|\Gamma_k \|, 1 \} \|J_k^TF_k\| \leq \max \{ M_0^{1-r}, M_0/\underline{\lambda}, M_0\} =: \bar{M}_0, \quad \forall k \in {\cal K}.
			\end{equation*}
			and
			\begin{equation*}%\label{norm_d_k1}
				\|d_k\| \leq \frac{1}{\gamma} \max \{ \bar{M}_0, M_0\} =: M_2, \quad \forall k \in {\cal K}.
			\end{equation*}
			Thus, the proof is complete. \qed 
		\end{proof}

%%%					Inequalities \eqref{eq:angle} and \eqref{norm_d_k1} show that for $k \in {\cal K}$, $d_k$ is the one determined in Step 2 of Algorithm~\ref{alg:LMM_armijo} because it passes the seize of Step~4 for $M=M_2$ and $\theta = \gamma \min \{1, \underline{\lambda}\}/M_1^2$.

			The above proposition shows that \emph{directions $d_k$ generated at Step 2} of Algorithm~\ref{alg:LMM_armijo} are gradient-related, 
			for both choices $\lambda_k = \| J_k^T F_k \|^r$, or $0 < \underline{\lambda} \leq \lambda_k \leq \bar{\lambda}$, 
			as long as the completeness condition~\eqref{comp-condition} holds at every iterate. 
			Nevertheless, apart from special cases, e.g when the Jacobian $J(x)$ is full-rank everywhere, 
			we cannot ensure that such condition holds at every $x_k$ generated by Algorithm~\ref{alg:LMM_armijo} 
		 starting at an arbitrary $x_0 \in \Rn$. 			
			Hence, we safeguard the algorithm by using the classic LMM direction whenever the LMMSS direction fails to fulfill sufficient conditions for gradient-related property, namely, those of Remark~\ref{obs:sufficient}. 
			%$\| d_k \| \leq M$, for $M>0$ and $\nabla \phi (x_k)^T d_k \leq -\theta \|\nabla \phi (x_k) \|^2$ , for  $ \theta \in ]0,1[$.
			
			If such conditions are not satisfied, we take $d_k = -(J_k^T J_k + \lambda_k I)^{-1} J_k^T F_k$ which yields gradient-related directions. In fact, from
			\[
			(J_k^T J_k + \lambda_k I) d_k = - \nabla \phi(x_k),
			\]
			we obtain
			\[
			- \nabla \phi(x_k)^T d_k = \| J_k d_k \|^2 + \lambda_k \| d_k \|^2 \geq  \lambda_k \| d_k \|^2,
			\]
			and since
			\[
			\| \nabla \phi(x_k) \| = \| (J_k^T J_k + \lambda_k I) d_k \| \leq \| J_k^T J_k + \lambda_k I\| \| d_k \| \leq M_3 \| d_k \|,
			\]
			where $M_3 = L_J^2 + \bar{\lambda}$. We conclude that
			\[
			- \nabla \phi(x_k)^T d_k \geq  \lambda_k \| d_k \|^2 \geq \frac{ \lambda_k }{M_3^2} \| \nabla \phi(x_k) \|^2 =  \frac{ 1 }{M_3^2} \| \nabla \phi(x_k) \|^{2+r}
			\]
			if we take $\lambda_k = \| \nabla \phi(x_k) \|^r$, or $- \nabla \phi(x_k)^T d_k \geq \frac{ \underline{\lambda} }{M_3^2} \| \nabla \phi(x_k) \|^2$, 
			if $\lambda_k \geq \underline{\lambda} > 0$. 
			Moreover, 
\begin{align*}
			\| d_k \| = \| (J_k^T J_k + \lambda_k I)^{-1} \nabla \phi(x_k) \| & \leq \| (J_k^T J_k + \lambda_k I)^{-1} \| \| \nabla \phi(x_k) \| \\
			\ & \leq \frac{1}{\lambda_k} \| \nabla \phi(x_k) \| = \| \nabla \phi(x_k) \|^{1-r},
\end{align*}
if $\lambda_k = \| \nabla \phi(x_k) \|^r$, or $\| d_k \| \leq \| \nabla \phi(x_k) \| / \underline{\lambda}$ if $\lambda_k \geq \underline{\lambda} > 0$. 
In both cases the right rand sides are bounded in the convergent subsequence $\{x_k\}_{k \in {\cal K}}$.

		Now that it is proven that Algorithm~\ref{alg:LMM_armijo} generates a sequence of directions $\{d_k\}$ that is gradient-related, 
		using \cite[Proposition~1.2.1]{Bertsekas2016}, global convergence can be established.
		
		\begin{teo}\label{teo:prop_conv}
		 Given $x_0 \in \Rn$, 
		 %suppose that condition~\eqref{comp-condition}  holds in the level set ${\cal L}_0 = \{ x \in \Rn \mid \phi(x) \leq \phi(x_0)\}$.
			let $\{ x_k \}$ be a sequence generated by Algorithm~\ref{alg:LMM_armijo} 
			with either $\lambda_k = \|J_k^T F_k\|^r, r \in ]0,1[$ or $0 < \underline{\lambda} \leq \lambda_k \leq \bar{\lambda}$.  
			Then, every limit point $\hat{x}$ of $\{ x_k \}$ is such that $\nabla \phi (\hat{x}) = 0$.
		\end{teo}
		\begin{proof}
			Let $K_1 = \{ k \in \mathbb{N} \mid \|J(x_k + d_k)^TF(x_k + d_k) \| \leq \vartheta \| J(x_k)^TF(x_k) \| \}$.
			If $K_1$ is infinite, it follows that $\| J(x_k)^TF(x_k)\| \rightarrow 0$, and therefore any limit point $\hat{x}$ of $\{x_k\}$ is such that $J(\hat{x})^TF(\hat{x})=0$, hence $\nabla \phi(\hat {x})=0$.
			Otherwise, if $K_1$ is finite, let us assume, without loss of generality, that $\|J(x_k+d_k)^TF(x_k + d_k) \| > \vartheta \| J(x_k)^TF(x_k) \|$, for each $k$, such that the step size is chosen to satisfy the Armijo condition. 
%			Since, by Proposition~\ref{prop_dk} and Remark~\ref{obs:dLM}, 
			Since the directions of Algorithm~\ref{alg:LMM_armijo} are gradient-related (due to Proposition~\ref{prop_dk} and discussion afterwards), 
			it follows from \cite[Proposition~1.2.1]{Bertsekas2016} that any limit point $\hat{x}$ of $\{x_k\}$ is a stationary point of $\phi(x)$. \qed 
		\end{proof}
		
		The next result establishes the connection between global and local convergence by showing that, under certain conditions, 
		$\alpha_k=1$ for all $k$ sufficiently large and then Algorithm~\ref{alg:LMM_armijo} becomes the ``pure'' LMMSS and the results of Section~\ref{sec:local} are applicable. This is the only result in this section which requires Assumption~\ref{hipotese5paraconvl}. 
		
		\begin{teo}\label{teo:globallocal}
			Let $\{ x_k \}$ be generated by Algorithm~\ref{alg:LMM_armijo}, using $x_0 \in \Rn$ as the initial point and $\lambda_k = \|J_k^T F_k\|$.
			Suppose that Assumption~\ref{Hip_error_boundnew}, Assumption~\ref{hipotese5paraconvl} with $r=1$, 
			and condition~\eqref{comp-condition} hold at the level set ${\cal L}_0 = \{ x \in \mathbb{R}^n \mid \phi(x) \leq \phi(x_0)\}$. 
			Moreover, assume that ${\cal L}_0$ is compact. 
			
			Then, for every $k$ sufficiently large, $\alpha_k=1$, and the sequence $\{ \dist(x_k,X^*) \}$ converges quadratically to zero.
		\end{teo}
		
		\begin{proof}
			Note that by the definition of Algorithm~\ref{alg:LMM_armijo} $\{ \phi(x_k) \}$ is decreasing. 
			Starting from $x_0$, we conclude that the sequence $\{x_k\}$ remains in ${\cal L}_0$, a compact set. 
			%Therefore, there exists a convergent subsequence, as per Theorem~\ref{teo:prop_conv}, where every limit point of this subsequence is stationary. By removing this convergent subsequence, we can repeat this process successively for the remaining sequence, which remains in a compact set, thereby concluding that $\{x_k\}$ is the disjoint and enumerable union of convergent subsequences, each of which has a stationary point as limit.
			By Theorem~\ref{teo:prop_conv}, every limit point $\hat{x}$ of $\{x_k\}$ is stationary, thus for $k$ sufficiently large, the norm of the gradient at $x_k$ will be sufficiently small. 
			In other words, there is a positive integer $k_0$ such that    
			\begin{equation}\label{theorem_2_auxB}
				\|J(x_{k})^TF(x_{k})\|\leq \frac{\omega^3 \vartheta}{\tilde{C} L_3},  \quad \text{for all} \quad k \geq k_0.
			\end{equation}
			
			For $k \geq k_0$, let $\bar{x}_{k+1} \in X^*$ such that $\|x_{k+1} -\bar{x}_{k+1} \| = \dist (x_{k+1},X^*)$. 
			Then, from \eqref{eqhipL3new}, Assumption~\ref{Hip_error_boundnew},\eqref{eq:hip5convlocal} and \eqref{theorem_2_auxB}, it follows that
			\begin{align*}
				\|J(x_{k+1})^TF (x_{k+1})\| &= \|J(x_{k+1})^TF (x_{k+1})- J(\bar{x}_{k+1})^TF(\bar{x}_{k+1})\|\\ \notag &\leq L_3 \|x_{k+1} -\bar{x}_{k+1} \|\\ \notag
				&= L_3 \dist (x_{k+1},X^*) \leq \frac{\tilde{C}}{\omega} L_3 \dist (x_{k},X^*)^2\\ 
				&\leq \frac{\tilde{C} L_3 \|J(x_k)^TF (x_k)\|}{\omega^3} \|J(x_k)^TF (x_k)\| \\ \notag
				&\leq \vartheta \|J(x_k)^TF (x_k)\|, %\label{theorem_2_aux2}
			\end{align*}
			proving that $\alpha_k = 1$ for each $k \geq k_0$. To obtain the quadratic convergence of  $\{ \dist (x_k, X^*) \}_{k \geq k_0}$ to 0, 
			one may now apply Theorem~\ref{teoconvlocalhip5CR} (or Theorem~\ref{teoconvlocalhip5rankdim}) 
			with $x^* = \hat{x}$ and $x_0 = x_{k_0}$. \qed 
		\end{proof}
		
		\begin{obs}
			We observe that, in the case of zero residual, a previous result \cite[Theorem 3.1]{Yamashita} assumed  that the limit point $x^*$ is such that $F(x^*) = 0$. Also, in the case of unconstrained optimization, in order to prove a result similar to Theorem~\ref{teo:globallocal} the required assumption is that the Hessian at $x^*$ is positive definite; however, such assumption would imply that $x^*$ is an isolated stationary point. 
			As we want to address the case of nonzero residual and possibly nonisolated stationary points, we considered the boundness of the level set instead of these other two conditions.
		\end{obs}
		
		\begin{obs}\label{obs:adjust}
			In the scenario of diminishing rank, if Assumption~\ref{hipotese5paraconvl} is satisfied with $0<r<1$, 
			it is necessary to adjust the parameter $\lambda_k$ of Algorithm~\ref{alg:LMM_armijo} to $\lambda_k=\norm{J_k^TF_k}^r$.
			 Furthermore, when this hypothesis holds, Theorem~\ref{teo:globallocal} follows 
			 with minor changes in some constants, 
			 and we can show there exists $k_0$ such that $\alpha_k=1, \forall k \geq k_0$ and that $\dist(x_k,X^*)$ goes to zero superlinearly.     
		\end{obs}

		\section{Illustrative examples}\label{sec:examples}
		The purpose of this section is to illustrate our convergence results by presenting 
		some examples of nonzero residual nonlinear least-squares problems.  
		In this section we use the shorthand notation (A1) -- (A4) for Assumptions~\ref{Hip_Jdiffnew}--\ref{hipotese5paraconvl}, respectively. 
		The examples came from reference \cite[Section~5]{behling2019local}, where  
		(A1)--(A4) were already demonstrated/verified and such verification will not be repeated here.\\
		
		\textbf{Example 1.} (Constant rank and (A4)) Let us start with an example from \cite[Example 5.1]{behling2019local} with nonisolated minimizers (stationary-points). 
		Consider the residual function 
		\[
		F(x) = (x_1^2 + x_2^2 - 1, x_1^2 + x_2^2 - 9)^T.
		\]
		The set of minimizers (stationary-points) $X^*$ consists of points satisfying $x_1^2 + x_2^2 - 5 = 0$, 
		thus $\text{dist}(x, X^*) = |\|x\| - \sqrt{5}|$. Moreover, for any $x^* \in X^*$, $\phi(x^*) = \frac{1}{2}\|F(x^*)\|^2= \frac{1}{2}(4^2+(-4)^2)= 16$. 
		
		It is straightforward to observe that 
		\[
		J(x) = \left[  \begin{array}{cc} 2x_1  & 2 x_2 \\ 2x_1  & 2 x_2 \end{array} \right]
		\] 
		has rank $1$ everywhere except for the origin. 
		Moreover, as discussed in \cite{behling2019local}, (A1), (A2) and (A4) with $r=1$ hold as well. 
		It is worth to point out that (A4) holds with $C=0$ because in this example $\| J(x)^T F(z) \| = 0$.
		
		By considering LMMSS with the scaling matrix $L=[-1 \quad 1]$,  condition \eqref{comp-condition} holds everywhere except for the line $x_2 = - x_1$. 
		
		Hence, if we consider an initial point  in a neighborhood of $x^* = (0, \sqrt{5})^T \in X^*$, 
		we are under conditions to apply the results of Section~\ref{subsec:constrank}. 
		As expected from Theorem~\ref{teoconvlocalhip5CR}, 
		in Table~\ref{tab:ex1convlocal}, we can observe the local quadratic convergence of $\text{dist}(x_k, X^*)$ to zero, 
		starting LMMSS from two distinct starting points in a neighborhood of $(0, \sqrt{5})^T$. 
		
		\begin{table}
			\centering
			\caption{Stopping criterion $\|J_k^T F_k\| < 10^{-8}$ with $x_0 = (0, \sqrt{5} + 0.03)^T$ as starting point (left), and $x_0 = (0.01, \sqrt{5} - 0.01)^T$ (right).}
			\label{tab:ex1convlocal}
			\begin{tabular}{cccccc}
				\hline
				$k$ & dist($x_k, X^*$) & $\norm{J_k^T F_k}$ & $k$ & dist($x_k, X^*$) & $\norm{J_k^T F_k}$ \\
				\hline
				0 & $1.3506 \times 10^{-1}$ & $1.2242$ & 0 & $4.4521\times 10^{-2}$ & $3.9643 \times 10^{-1}$ \\
				1 & $1.7762 \times 10^{-2}$ & $1.5890 \times 10^{-1}$ & 1 & $1.9821 \times 10^{-4}$ & $1.7729 \times 10^{-3}$ \\
				2 & $3.2402 \times 10^{-7}$ & $2.8982 \times 10^{-6}$ & 2 & $3.8598 \times 10^{-9}$ & $3.4523 \times 10^{-8}$ \\
				3 & $1.1546 \times 10^{-14}$ & $1.0659 \times 10^{-13}$ & & & \\
				\hline
			\end{tabular}
		\end{table}
		
		% efeito da L
		This example is also useful to illustrate the effect of using $L^T L \ne I$ and the globalization mechanisms of Algorithm~\ref{alg:LMM_armijo}.  
		Figure~\ref{fig:ex1_2_4_a} depicts the trajectory of LMM ($L^T L = I$) and the one of LMMSS (with $L = [-1 \quad 1]$) both starting from $x_0 = (2,4)^T$. 
		As we can see from the figure, while LMM dives perpendicularly to the set $X^*$ (observe the level curves in the figure), 
		LMMSS trajectory seems parallel to the null space of $L$,  
		and they end up at different stationary points but with the same residual norm. 
		Figure~\ref{fig:ex1_2_4_b} plots the gradient norm through the iterations of LMM and LMMSS.

%		As exposed in \cite{Boos2024} for parameter identification problems, 
%		an advantage of using Algorithm~\ref{alg:LMM_armijo} is that in a given problem with nonisolated solutions, we can choose $L$ that enforces a specific property. 
%		In this geometric example, by using $L = [-1, 1]$, we aim for the method to prioritize stationary points where $\| L(x - x_0) \| = |x_2 - x_1 - 2|$ is as small as possible. 
%		In \cite{Boos2024}, the authors employed $L$ as discrete versions of derivative operators in order to promote smoothness 
%		in the obtained solutions for parameter identification problems in heat conduction. 
		
%		Moreover, Figure~\ref{fig:ex1_2_4_b} enforces this aspect by showing that the presence of a well-thought choice of $L$ may conduce the iterations faster and in a more accurate manner than LMM, while it stimulates a desired property.
		
		\begin{figure}[!ht]
			\centering
			\subfloat[Iterations and level curves of the objective function. \label{fig:ex1_2_4_a}]{\includegraphics[scale=0.5]{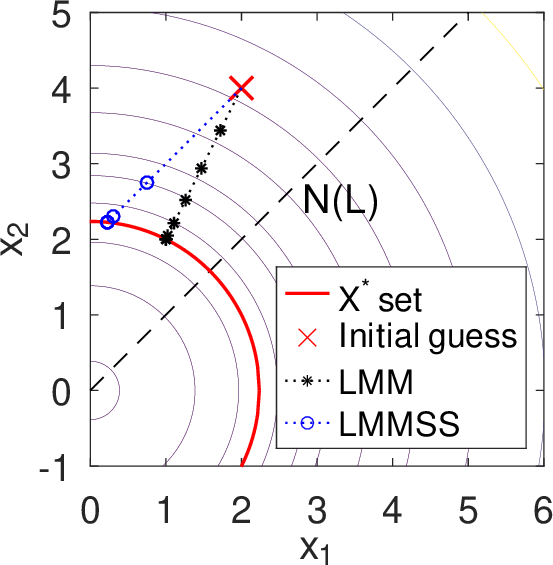}} \qquad \qquad 
			\subfloat[Values of $\|J_k^T F_k\|$ until the stopping criterion is achieved. \label{fig:ex1_2_4_b}]{\includegraphics[scale=0.5]{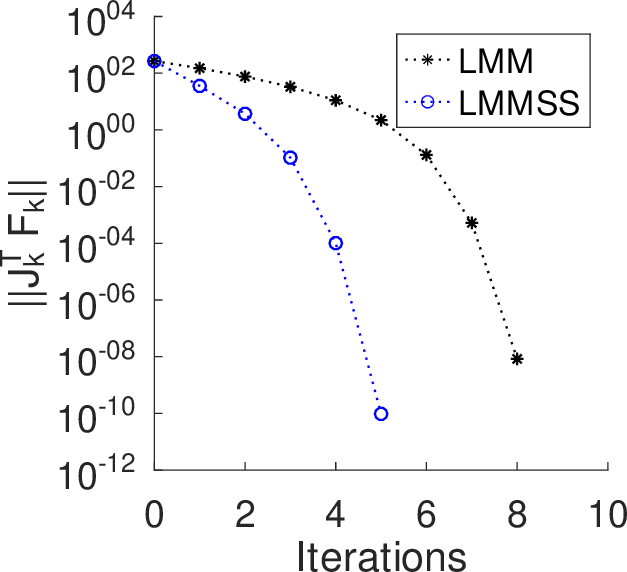}}
			\caption{Plot of iterations for LMM and LMMSS for Example 1, with starting point $(2,4)$.}
			\label{fig:ex1_2_4}
		\end{figure}
		
		The last aspect we shall discuss with this example is about the completeness condition \eqref{comp-condition} 
		and the global convergence mechanisms of Algorithm~\ref{alg:LMM_armijo}. 
		As mentioned before, condition ~\eqref{comp-condition} holds everywhere except for the line $x_2 = -x_1$. 
		LMMSS also tries to move parallel to the null space of $L$. Figure~\ref{fig:ex1_-1_3_a} illustrates what happens if we apply LMMSS without the safeguard Step~4, 
		using $x_0 = (-1, 3)^T$. The displacements $x_k - x_0$ are in ${\cal N}(L)$, but the hyperplane $x_0 + {\cal N}(L)$ does not intersect $X^*$. 
		As a result, LMMSS minimizes $\phi(x)$ in $x_0 + {\cal N}(L)$, approximating a point $\hat{x}$ where $\nabla \phi(\hat{x}) \perp {\cal N}(L)$ 
		so that the directions $d_k$ become orthogonal to $-\nabla \phi(x_k)$ loosing the descent property and the algorithm crashes not satisfying the Armijo condition. 
		Notice that $\hat{x}$ is such that $\hat{x}_2 = -\hat{x}_1$ and condition~\eqref{comp-condition} fails. 
		On the other hand, if we use Algorithm~\ref{alg:LMM_armijo} with the safeguard step, Step~4 corrects this behaviour and,
		 when $x_k$ approaches $\hat{x}$,  
		the classic LMM direction is employed instead; 
		see Figure~\ref{fig:ex1_-1_3_b}. This allows the algorithm to move on and reach $X^*$. 
		
		\begin{figure}[!ht]
			\centering
			\subfloat[LMMSS without safeguard Step 4 in Algorithm~\ref{alg:LMM_armijo}. \label{fig:ex1_-1_3_a}]{\includegraphics*[scale=0.49]{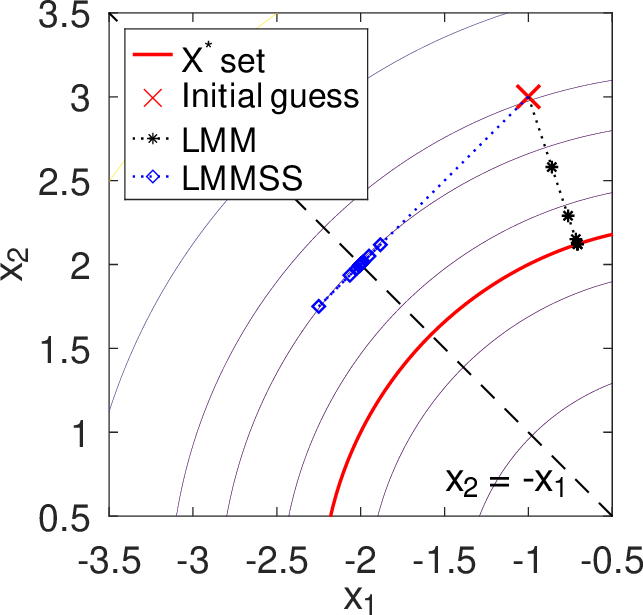}} \qquad \qquad 
			\subfloat[LMMSS with safeguard Step 4 in Algorithm~\ref{alg:LMM_armijo}. \label{fig:ex1_-1_3_b}]{\includegraphics*[scale=0.49]{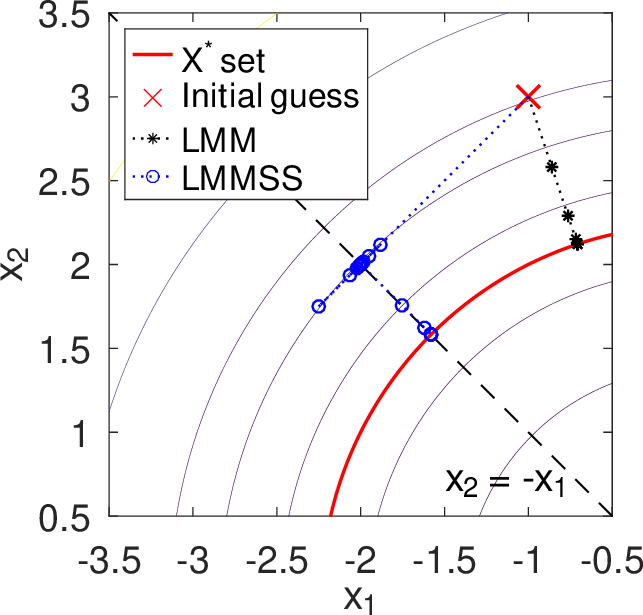}}
			\caption{Plot of iterations for LMM and LMMSS for Example 1, with starting point $(-1,3)$, and level curves of the objective function.}
			\label{fig:ex1_-1_3}
		\end{figure}

		\textbf{Example 2.} (Diminishing rank and (A4)) Now we consider an example to illustrate the ``diminishing rank'' scenario. 
		The residual function is given by 
		\[
		F(x) = (x_1^3 -x_1x_2 + 1, x_1^3 + x_1x_2 + 1),
		\] 
		which leads to $\phi(x)$ having an isolated global minimizer at $(-1,0)^T$ and a nonisolated set of local minimizers $\{x \in \R^2 \mid x_1 = 0\}$ 
		that we shall use as $X^*$ and, in this case, $\dist(x,X^*) = |x_1|$. 
		
		This example was also analyzed in \cite[Example~5.2]{behling2019local} 
		where it was verified that assumptions (A1), (A2) and (A4) with $r=1$ hold in a neighborhood of $x^* = (0,2)^T$. 
		Differently from the previous example where $J(x)^T F(z) = 0$, here (A4) holds because $\| J(x)^T F(z) \| = 6 x_1^2 = 6 \dist(x,X^*)^2$.
		
		For condition~\eqref{comp-condition}, we point out that the null space of
		\[
		J(x) = \left[  \begin{array}{cc} 3x_1^2 - x_2 & \quad -x_1 \\3x_1^2 + x_2  & \quad x_1 \end{array} \right]
		\]
		is $\{0\}$ wherever $x_1 \ne 0$ and it is $\text{span} \{(0,1)^T\}$ for $x \in X^* \setminus \{0\}$. 
		Hence, it is clear that, regardless the choice of $L$, condition~\eqref{comp-condition} holds everywhere except for $X^*$.
		
		In Table~\ref{tab:ex2} and Figure~\ref{fig:ex2} we report the behaviour of LMM and LMMSS, with $L = [-1 \quad 1 ]$,  
		starting in the neighbourhood of $x^* = (0,2)^T$. 
		From Table~\ref{tab:ex2}, the quadratic convergence rate is apparent for both methods. 
		In Figure~\ref{fig:ex2} we highlight the different limit points: 
		while the classic LMM moves almost perpendicularly to $X^*$, 
		LMMSS moves parallel to ${\cal N}(L)$, leading in this case to a stationary point with smaller $x_2$ coordinate. 
		
		\begin{figure}[!ht]
			\centering
			\includegraphics*[scale=0.5]{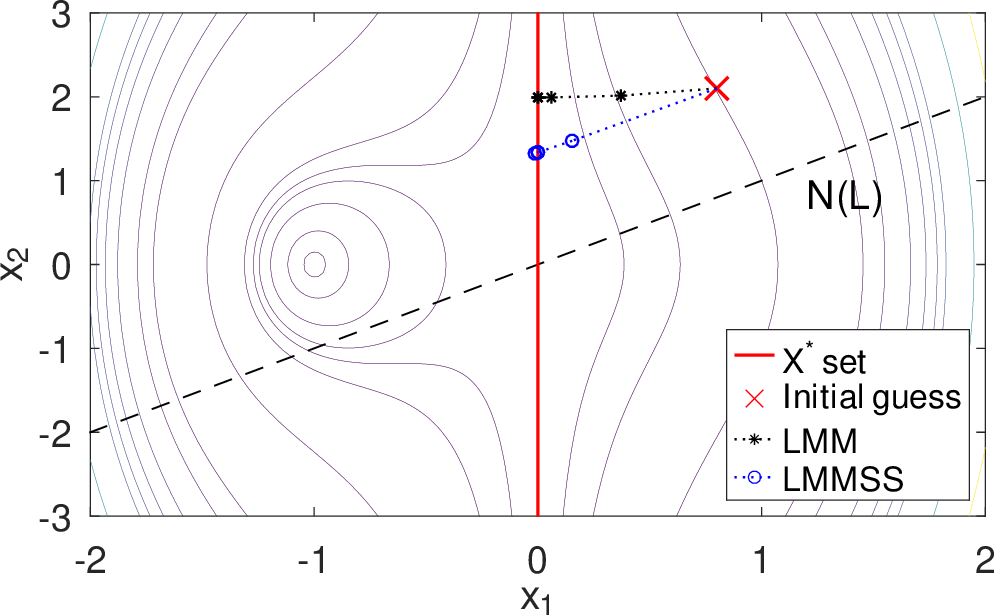}
			\caption{Plot of iterations for LMM and LMMSS for Example 2, with starting point $(0.8,2.1)^T$, and level curves of the objective function. Limit points: $(1.4833\times 10^{-14}, 1.9915)^T$ for LMM and $(-7.5482\times 10^{-16}, 1.3377)^T$ for LMMSS.}
			\label{fig:ex2}
		\end{figure}
		
		\begin{table}
			\centering
			\caption{Values of dist($x_k, X^*$) until the stopping criterion $\|J_k^T F_k\| < 10^{-10}$ is achieved for Example 2 with $x_0 = (0.8, 2.1)^T$.}
			\label{tab:ex2}
			\begin{tabular}{ccc}
				\hline
				$k$ & LMM & LMMSS  \\
				\hline
				1 & $3.7143 \times 10^{-1}$ & $1.5307 \times 10^{-1}$  \\
				2 & $6.0270 \times 10^{-2}$ & $1.3438 \times 10^{-2}$ \\
				3 & $1.0055 \times 10^{-3}$ & $1.7991 \times 10^{-4}$ \\
				4 & $2.4684 \times 10^{-7}$ & $3.0097 \times 10^{-8}$ \\
				5 & $1.4833 \times 10^{-14}$ & $7.5482 \times 10^{-16}$ \\
				\hline
			\end{tabular}
		\end{table}
 	\textbf{Example 3.} (Diminishing rank and (A3))	 The residual function
 	\[
 	F(x) = (\frac{1}{9} \cos x_1 - x_2 \sin x_1, \ \ \frac{1}{9} \sin x_1 + x_2 \cos x_1 )^T
 	\]
 	has a set of nonisolated minimizers $X^* = \{ x \in \R^2 \mid x_2 = 0 \}$, thus $\dist(x,X^*) = |x_2|$. 
 	In \cite[Example~5.3]{behling2019local} it was verified that for $x^* = (\pi, 0)^T$, and $\delta = 0.1$, 
 	(A1) is satisfied, as well as (A2) with $\omega = 1$ (because $\| \nabla \phi(x) \| = |x_2|$).

 	The Jacobian, given by
 	\[
 	J(x) = \left[
 	\begin{array}{cc}
 	-x_2  \cos x_1 - \frac{1}{9}\sin x_1 &    \quad -\sin x_1 \\
 	\ \\
 	-x_2 \sin x_1  + \frac{1}{9} \cos x_1  &  \quad \cos x_1
 	\end{array}
 	\right]
 	\]
 	has full rank for $x_2 \ne 0$ and rank 1 for $x_2 = 0$. Thus, the rank of $J(x_k)$ decreases as $x_k$ approaches $X^*$ from the outside. 
 	Moreover, condition~\eqref{comp-condition} is valid in $B(x^*,\delta) \setminus X^*$. 
 	
 	 	Unfortunately, (A4) does not hold in this example. But (A3) does, because
\begin{align*}
 	\| J(x)^T F(z) \| & = \frac{1}{18} \sqrt{81 \sin^2 (x_1 - z_1) + (9 x_2 \cos(x_1-z_1) + \sin(x_1 - z_1))^2} \\
 	& \leq  \frac{\sqrt{91}}{81} \| x - z \|,
\end{align*}
 	for $x \in B(x^*, \delta)$ and $z \in X^* \cap B(x^*,\delta)$. Hence, $\sigma = \sqrt{91}/81  \approx 0.118$. 
 	
 	Other constants relevant for the analysis are given approximately by 
 	$L_0 = 1.1156$, $L_4 = 0.5612$, $\gamma = 0.8949$, $(\sigma_{\min}^*)^2 = 1.0123$. 
 	We observe that $\sigma < (\sigma_{\min}^*)^2$, and for $\theta = 1.0315$, we ensure $ (\sigma_{\min}^*)^2 > \theta \sigma$. 
 	In this case, $\bar{\delta} = 0.1 = \delta$, and assumptions required by Lemma~\ref{lemac1postodiminuindo} hold. 

 	By considering $L = [1 \quad 0]$, $\| L \| = 1$ and from \eqref{c1compostodiminuindoa3} with $\kappa = 1.0001$ we obtain $c_1 = 2.6172$. 
 	Then, since $\bar{\sigma} = 0.1367 > \sigma$, we can follow Lemma~\ref{lem:Auxconv1a3} with $\varepsilon <= 0.0016$ 
 	and $\eta \in ] 0.8617, 1 [$. 
 	
 	Setting $\eta = 0.9$, makes $\varepsilon = 0.0016$ in \eqref{eq:epsA3DIMRANK} and 
 	starting from $x_0 = (\pi, 0.001)^T$, and using  $\lambda_k = \sigma + L_4 \|x_k - \bar{x}_k \|$, 
 	Theorem~\ref{teoconvlocalA3rankdim} applies, and we expect linear convergence, as can be verified in Table~\ref{tab:ex3}. 
 	
 	\begin{table}
 	\centering
 	\begin{tabular}{lr}
 	\hline 
 	$k$ & $\dist(x_k,X^*)$ \\ \hline 
 	 0 & 0.001 \\ 
 1 & 0.00010413  \\
 2 & 1.0889e-05 \\
 3 &  1.1392e-06 \\
 4 &  1.1919e-07 \\
 5 &  1.247e-08 \\
 6 &  1.3046e-09 \\
 7 &  1.365e-10\\
 8 &  1.4281e-11 \\ \hline 
 	\end{tabular}
 	\caption{Iterations of LMMSS for Example~3, initial point $x_0 = (\pi, 0.001)^T$.}\label{tab:ex3}
 	\end{table}

		\section{Final remarks}\label{sec:conclusion}
		
		We have investigated the local convergence of Levenberg-Marquardt method with singular scaling (LMMSS) proposed in \cite{Boos2024} 
		when applied to nonlinear least-squares problems with \emph{nonzero} residual. 

		Our study reveals that, 
		regardless the Jacobian rank is constant or not in a neighborhood of $x^* \in X^*$,  
		the convergence rate of $\dist(x_k,X^*)$ to zero depends on a combined measure of nonlinearity and residual size $\| J(x)^T F(z) \|$ (see Assumption~\ref{a3} and \ref{hipotese5paraconvl}) and suitable choices of the LM parameter $\lambda_k$ (Theorems~ \ref{teoconvlocahip4}--\ref{teoconvlocalA3rankdim}). 
		For $\| J(x)^T F(z) \| = o(\| x - z\|)$, the local convergence rate is superlinear whereas $\| J(x)^T F(z) \| \leq \sigma \| x - z\|$ leads to linear convergence, 
		but only when $\sigma$ is small enough relatively to the error bound constant $\omega$. 
		These results are aligned with those of \cite{behling2019local} but consider the more general case where $L^T L$ can be singular, 
		as long as the uniform completeness condition (Definition~\ref{def:ucomp}) holds in a suitable neighborhood of $x^* \in X^*$. 
		
		Furthermore, through  a few modifications in the algorithm proposed in  \cite[Algorithm~1]{Boos2024}, 
		namely, the choice of LM parameter, 
		the full step acceptance criterion 
		and the use of classic LMM direction as a safeguard whenever LMMSS fails to satisfy descent properties; 
		we have established under mild assumptions the global convergence of Algorithm~\ref{alg:LMM_armijo}: 
		every limit point of the generated sequence is stationary for the nonzero residual nonlinear least squares problem. 
		
		In future works we will look for weaker assumptions to control $\| J(x)^T F(z) \|$. 
		The condition $\| J(x)^T F(z) \| \leq \sigma \| x - z\|$ with small enough $\sigma$ 
		may be too restrictive and hold only in nonzero residual problems with very small residual norm. 
		When it does not hold, we cannot even ensure that  the ``pure'' LMMSS converges locally. 
		In that case, other methods such as Quasi-Newton or even Newton, when the Hessians of $F_i(x)$ are available, 
		might be more appropriate.

	\section*{Acknowledgments}
	This work was supported by Brazilian agencies FAPESC (Fundação de Amparo à Pesquisa e Inovação do Estado de Santa Catarina) [grant number 2023TR000360] and CNPq (Conselho Nacional de Desenvolvimento Científico e Tecnológico) [grant number 305213/2021-0].
	
	%%\section*{Appendix} 
	%%
	%%In the following, we show that in certain situations  
	%%the LM parameter $\lambda_k > 0$ plays no role in determining the direction $d_k$. 
	%%
	%%\begin{prop}
	%%If $p < n \leq m$, $\text{rank}(J_k) = n - p$, $J_k^T F_k \ne 0$ and Assumption~\ref{Hip_nullJLnew} holds, then any $d_k$ which minimizes 
	%%\[
	%%\| J_k d + F_k \|^2 + \lambda_k \| L d \|^2
	%%\]
	%%is such that $d_k \in {\cal N}(L)$, regardless the choice of $\lambda_k > 0$.
	%%\end{prop}
	%%
	%%
	%%\begin{proof}
	%%Any minimizer of $\| J_k d + F_k \|^2$ is such that 
	%%\[
	%%J_k^T J_k d = - J_k^T F_k, 
	%%\]
	%%which defines an affine subspace 
	%%\[
	%%S_k = \{ d \in \Rn \mid J_k^T J_k d = - J_k^T F_k \}
	%%\]
	%%parallel to ${\cal N}(J_k^T J_k) = {\cal N}(J_k)$. 
	%%On the other hand, due to Assumption~\ref{Hip_nullJLnew}, ${\cal N}(J_k) \cap {\cal N}(L) = \{0\}$. 
	%%Since $\text{dim}{\cal N}(J_k) = p$ and $\text{dim} {\cal N}(L) = n - p$ we have $\Rn = {\cal N}(J_k) \oplus {\cal N}(L)$ 
	%%wich implies that  $S_k \cap {\cal N}(L) \ne \emptyset$. 
	%%Therefore, any $d \in S_k \cap {\cal N}(L)$ minimizes both terms $\| J_k d + F_k \|^2$ and $\lambda_k \| L d \|^2$ independently of $\lambda_k > 0$.
	%%\end{proof}
	%%
	%%\textbf{Remark.} The condition $S_k \cap {\cal N}(L) \ne \emptyset$ does not hold when $J_k$ is full rank in which case $S_k$ is a singleton -- and nonzero vector because $J_k^T F_k \ne 0$. When $\text{rank}(J_k) > n - p$, that is, $\text{dim}{\cal N}(J_k) < p$ the claim of the above proposition holds if, and only if
	%%\[
	%%J_k^{+} F_k \in {\cal N}(J_k) \oplus {\cal N}(L) \subsetneq \Rn.
	%%\]

	\bibliographystyle{plain}
	\bibliography{refs}

\begin{thebibliography}{10}

\bibitem{barz2015nonlinear}
Tilman Barz, Stefan K{\"o}rkel, G{\"u}nter Wozny, et~al.
\newblock Nonlinear ill-posed problem analysis in model-based parameter
  estimation and experimental design.
\newblock {\em Computers \& Chemical Engineering}, 77:24--42, 2015.

\bibitem{behling2012}
R.~Behling and A.~Fischer.
\newblock A unified local convergence analysis of inexact constrained
  levenberg--marquardt methods.
\newblock {\em Optimization Letters}, 6(5):927--940, 2012.

\bibitem{behling2019local}
R.~Behling, D.~S. Gon{\c{c}}alves, and S.~A. Santos.
\newblock Local convergence analysis of the levenberg--marquardt framework for
  nonzero-residue nonlinear least-squares problems under an error bound
  condition.
\newblock {\em Journal of Optimization Theory and Applications},
  183(3):1099--1122, 2019.

\bibitem{behling2013}
R.~Behling and A.~Iusem.
\newblock The effect of calmness on the solution set of systems of nonlinear
  equations.
\newblock {\em Math. Program.}, 137(1-2):155--165, 2013.

\bibitem{bellavia2018elliptical}
Stefania Bellavia and Elisa Riccietti.
\newblock On an elliptical trust-region procedure for ill-posed nonlinear
  least-squares problems.
\newblock {\em Journal of Optimization Theory and Applications}, 178:824--859,
  2018.

\bibitem{Bergou2020}
E.~Bergou, Y.~Diouane, and V.~Kungurtsev.
\newblock Convergence and complexity analysis of a levenberg--marquardt
  algorithm for inverse problems.
\newblock {\em Journal of Optimization Theory and Applications},
  185(3):927--944, 2020.

\bibitem{Bertsekas2016}
D.~P. Bertsekas.
\newblock {\em Nonlinear Programming}.
\newblock Athena Scientific, third edition, 2016.

\bibitem{Boos2024}
E.~Boos, D.~S. Gon\c{c}alves, and F.~S.~V. Baz\'{a}n.
\newblock Levenberg-marquardt method with singular scaling and applications.
\newblock {\em Applied Mathematics and Computation}, 474:128688, 2024.

\bibitem{Buccini2017}
A.~Buccini, M.~Donatelli, and Reichel L.
\newblock Iterated tikhonov regularization with a general penalty term.
\newblock {\em Numer Linear Algebra Appl.}, 24:e2089, 2017.

\bibitem{cornelio2011regularized}
Anastasia Cornelio.
\newblock Regularized nonlinear least squares methods for hit position
  reconstruction in small gamma cameras.
\newblock {\em Applied mathematics and computation}, 217(12):5589--5595, 2011.

\bibitem{deidda2014regularized}
Gian~Piero Deidda, Caterina Fenu, and Giuseppe Rodriguez.
\newblock Regularized solution of a nonlinear problem in electromagnetic
  sounding.
\newblock {\em Inverse Problems}, 30(12):125014, 2014.

\bibitem{Dennis_Schnabel}
J.~{Dennis Jr.} and R.~Schnabel.
\newblock {\em Numerical Methods for Unconstrained Optimization and Nonlinear
  Equations}, volume~16 of {\em Classics in Applied Mathematics}.
\newblock Society for Industrial and Applied Mathematics (SIAM), Philadelphia,
  PA, 1996.

\bibitem{Dennis1977}
J.~E. {Dennis Jr.}
\newblock Nonlinear least squares and equations.
\newblock In D.~Jacobs, editor, {\em The State of the Art in Numerical
  Analysis}, pages 269--312. Academic Press, London, 1977.

\bibitem{engl}
H.~W. Engl, M.~Hanke, and A.~Neubauer.
\newblock {\em Regularization of {I}nverse {P}roblems}.
\newblock Kluwer Academic Publishers, 1996.

\bibitem{fan}
J.~Fan.
\newblock Convergence rate of the trust region method for nonlinear equations
  under local error bound condition.
\newblock {\em Comput. Optim. Appl.}, 34(2):215--227, 2006.

\bibitem{FanYuan}
J.~Fan and Y.~Yuan.
\newblock On the quadratic convergence of the levenberg--marquardt method
  without nonsingularity assumption.
\newblock {\em Computing}, 74(1):23--29, 2005.

\bibitem{Fischer2002}
A.~Fischer.
\newblock Local behavior of an iterative framework for generalized equations
  with nonisolated solutions.
\newblock {\em Math. Program.}, 94(1):91--124, 2002.

\bibitem{Fischer2024}
A.~Fischer, A.~Izmailov, and M.~V. Solodov.
\newblock The levenberg--marquardt method: an overview of modern convergence
  theories and more.
\newblock {\em Computational Optimization and Applications}, 89:33--67, 2024.

\bibitem{Gratton2007}
S.~Gratton, A.~S. Lawless, and N.~K. Nichols.
\newblock Approximate gauss--newton methods for nonlinear least squares
  problems.
\newblock {\em SIAM J. Optim.}, 18:106--132, 2007.

\bibitem{Hansen1998}
P.~C. Hansen.
\newblock {\em Rank-Deficient and Discrete Ill-Posed Problems}.
\newblock SIAM, Philadelphia, 1998.

\bibitem{Hansen2010}
P.~C. Hansen.
\newblock {\em Discrete Inverse Problems - Insight and Algorithms}.
\newblock SIAM, Philadelphia, 2010.

\bibitem{henn2003levenberg}
Stefan Henn.
\newblock A levenberg--marquardt scheme for nonlinear image registration.
\newblock {\em BIT Numerical Mathematics}, 43:743--759, 2003.

\bibitem{landi2017limited}
Germana Landi, E~Loli Piccolomini, and JG~Nagy.
\newblock A limited memory bfgs method for a nonlinear inverse problem in
  digital breast tomosynthesis.
\newblock {\em Inverse Problems}, 33(9):095005, 2017.

\bibitem{Levenberg1944}
K.~Levenberg.
\newblock A method for the solution of certain non-linear problems in least
  squares.
\newblock {\em Quarterly of Applied Mathematics}, 2:164--168, 1944.

\bibitem{Marquardt1963}
D.~Marquardt.
\newblock An algorithm for least-squares estimation of nonlinear parameters.
\newblock {\em SIAM Journal on Applied Mathematics}, 11(2):431--441, 1963.

\bibitem{Morozov1984}
V.~A. Morozov.
\newblock {\em Regularization {M}ethods for {S}olving {I}ncorrectly {P}osed
  {P}roblems}.
\newblock Springer, 1984.

\bibitem{ostrowski}
A.M. Ostrowski.
\newblock A quantitative formulation of {S}ylvester's law of inertia.
\newblock {\em Proc. Natl. Acad. Sci. USA}, 45(5):740--744, 1959.

\bibitem{pes2020}
F.~Pes and G.~Rodriguez.
\newblock The minimal-norm {G}auss-{N}ewton method and some of its regularized
  variants.
\newblock {\em Electronic Transactions on Numerical Analysis}, 53:459--480,
  2020.

\bibitem{Pes2022}
Federica Pes and Giuseppe Rodriguez.
\newblock A doubly relaxed minimal-norm gauss--newton method for
  underdetermined nonlinear least-squares problems.
\newblock {\em Applied Numerical Mathematics}, 171:233--248, 2022.

\bibitem{Tikhonov95}
A.~N. Tikhonov, A.~V. Goncharsky, V.~V. Stepanov, and A.~G. Yagola.
\newblock {\em Numerical methods for the solution of ill-posed problems},
  volume 328.
\newblock Springer Science \& Business Media, Dordrecht, 1 edition, 1995.

\bibitem{Yamashita}
N.~Yamashita and M.~Fukushima.
\newblock On the rate of convergence of the {L}evenberg-{M}arquardt method.
\newblock In {\em Topics in numerical analysis}, pages 239--249. Springer,
  2001.

\end{thebibliography}
\end{document}